\documentclass{amsart}
\usepackage{multirow}
\usepackage{rotating}
\usepackage{amsthm}
\usepackage{amsfonts}
\usepackage[UKenglish]{babel}
\usepackage{color}
\usepackage{cite}
\usepackage{graphicx}
\usepackage{soul}
\usepackage{marginnote}
\usepackage{subfig}
\usepackage{fullpage}
\usepackage{morefloats}
\newtheorem{theoremm}{Theorem}
\newtheorem{corol}[theoremm]{Corollary}
\newtheorem*{theorem}{Theorem}
\newtheorem{assumptionB}{Biological Assumption}
\newtheorem{assumptionM}{Mathematical Assumption}
\newtheorem*{Bass3S}{Biological Assumption 3 for serine recombinases}
\newtheorem*{Bass3T}{Biological Assumption 3 for tyrosine recombinases}
\newtheorem*{Mass3S}{Mathematical Assumption 3 for serine recombinases}
\newtheorem*{Mass3T}{Mathematical Assumption 3 for tyrosine recombinases}
\theoremstyle{lemma}
\newtheorem{lemma}{Lemma}
\theoremstyle{plain}
\newtheorem*{claim}{Claim}
\theoremstyle{definition}
\newtheorem*{definition}{Definition}
\theoremstyle{remark}
\newtheorem{case}{\textit{Case}}

\begin{document}

\title{Characterization of Knots and Links Arising from Site-specific Recombination on Twist Knots}

\author {Dorothy Buck$^1$ and Karin Valencia$^2$}
\date{\today}
\maketitle
\begin{center}
\small{Imperial College London, Department of Mathematics\\Postal address: Imperial College London, South Kensington Campus, Department of Mathematics,  London SW7 2AZ, England
\\(1) e-mail: d.buck@imperial.ac.uk, Telephone number: +44 (0)20 758 58625\\(2) (corresponding author) e-mail: karin.valencia06@imperial.ac.uk}
\end{center}

\begin{abstract}
 We develop a model characterizing all possible  knots and links arising from recombination starting with a twist knot substrate, extending previous work of Buck and Flapan.   We show that all knot or link products fall into three well-understood families of knots and links, and prove that given a positive integer $n$, the number of product knots and links with minimal crossing number equal to $n$ grows proportionally to $n^5$. In the (common) case of twist knot substrates whose products have minimal crossing number one more than the substrate, we prove that the types of products are tightly prescribed.  Finally, we give two simple examples to illustrate how this model can help determine previously uncharacterized experimental data.
\end{abstract}


\section{Introduction} \label{intro}

\noindent The central axis of the famous DNA double helix can become knotted or linked as a result of numerous biochemical processes, most notably site-specific recombination \cite{MaxwellBates, wang, mobileDNA}.  A  wide variety of DNA knots and links have been observed  \cite{vinograd, Wasserman-Cozzarelli1985, Cozzarelli-et-al,4,5,6,7,8,9,10,11,12,13}.  Characterising the precise knot or link type can often help understand structural or mechanistic features of the biochemical reaction \cite{ES, darcy, Sumners-et-al, 10, vazquez-gin, vazquez-et-al, d3, d4, v1, k1, k2 ,d5}.
\\\\Experimentally, such a characterization is typically achieved via gel electrophoresis (which stratify DNA products according to their minimal crossing number) \cite{15} and electron microscopy (which allows us to visualize the over- and under-crossings of the DNA molecule) \cite{14,16} together with knot invariants such as the Jones polynomial (amongst many others) \cite{jonespoly}. However, electron microscopy is not straightforward and often the precise over- or under-crossing cannot be categorically determined. Partial information can be gleaned by using gel electrophoresis but as there are 1,701,936 prime knots with minimal crossing number less than 17 this information is not sufficient \cite{numberofknots}. Furthermore, gel electrophoresis does not distinguish between handedness of chiral knots, so this does not give the full picture.  Thus topological techniques, such as those presented here, can aid experimentalists in characterizing DNA knotted and linked molecules by restricting the types of knots or links that can arise in a particular context.
\\\\Here we focus on the most common biochemical reaction that yields DNA knots and links: \textit{site-specific recombination}. Site-specific recombination is an important cellular reaction that has been studied extensively since the 1960s. It involves a reciprocal exchange between defined DNA segments. Biologically, this results in a variety of processes (see \cite{bio} and references therein). Apart from their fundamental functions in the cell, site-specific recombinases give scientists an elegant, precise and efficient way to insert, delete, and invert segments. Thus they are rapidly becoming of pharmaceutical and agricultural interest as well as  being used in the development of biotechnological tools \cite{mouse, tools}.
\\\\Twist knots are one of the most common DNA conformations. This is not surprising as in the cell most DNA is (plectonemically) supercoiled (like an over-used phone cord) and in the lab most experiments done with site-specific recombinases use small (plectonemically) supercoiled circular DNA molecules, so an unknot can be transformed to a twist knot by a single crossing change (see Figure \ref{supercoiledtotwist}).  
Unlike $(2,n)$-torus knots, twist knots occur as knots (not links) for both odd and even minimal crossing number, MCN$(K)$.  Thus ubiquitous DNA twist knots arise as a result of a variety of site-specific recombination reactions \cite{mobileDNA,Wasserman-Cozzarelli1985,Cozzarelli-et-al,4,5,6,7,8}.  
\\\\
Despite the biological importance of this twist knot family, there has yet to be a systematic model incorporating these as substrates for a generic site-specific recombinase.   (Earlier predictions of knots arising from site-specific recombination did not consider twist knots \cite{BFmaths, BFbio,SECS}).  Here we rectify this by presenting a model, extending the work of \cite{BFmaths}, classifying all possible knots and links that can arise from site-specific recombination on a twist knot. 
\\\\Our model is built on three assumptions for which biological evidence is provided in \cite{BFbio, KDbio}. We  construct a model that predicts
all possible knots and links that can arise as products of a single round of recombination, multiple rounds of
(processive) recombination, and of distributive recombination, given a twist knot
substrate $C(2, v)$ and our three assumptions. We predict that products arising from site-specific recombination on a twist knot substrate $C(2, v)$ must be members of the three families of products illustrated in Figure \ref{alternative family}. Members of these families
of knots and links include prime and composite knots and links with up to three components (see Section 2.3). Our model can
also distinguish between the chirality of the  product molecules of site-specific recombination (see Section 5). Our model is independent of
site orientation, and we make no assumption on the number of base pairs of the molecule(s).

\subsection{Structure of our paper}
\noindent This paper is organised as follows: in Section \ref{biodefs} we give a concise introduction to site-specific recombination and introduce notation. In Section \ref{assumptions} we state and explain the three assumptions about the recombinase complex, the substrate and the mechanisms of recombination. (Biological justifications for these assumptions can be found in \cite{BFbio,KDbio}). In Section \ref{lemmas} we determine the pre-recombinant and post-recombinant conformations of the recombination sites and all possible conformations of the DNA-protein complex; we also prove the necessary background lemmas for Section \ref{theorems}. In Section \ref{theorems} we prove Theorems 1 and 2 which determine all the putative DNA knot and link products of (non-distributive) site-specific recombination on a twist knot substrate. We show that these products belong to one of the three families of knots and links illustrated in Figure \ref{alternative family}. (These families of knots and links are defined in section 2.3). We also identify all knots and links that arise as products of distributive site-specific recombination. In Section \ref{MCN}, we prove Theorem 4 which shows that all the possible DNA knot and link products of site-specific recombination on a twist knot substrate are a very small fraction of all knots and links. We also further restrict the knot and link types of products that have minimal crossing number one more that of the substrate. 
Finally, in Section \ref{applications} we consider two simple uses of our model. For a detailed biological discussion of applications of our model, and how to use this model as a tool in a variety of site-specific recombination systems, we refer the reader to \cite{KDbio}.

\section{Biological systems and terminology}\label{biodefs}
\noindent In this section we give a concise introduction to site-specific recombination, introduce notation, and describe the families of knots and links that arise as products of site-specific recombination on twist knot substrates.

\subsection{Site-specific recombination}
\noindent Site-specific recombination reshuffles DNA sequences by inserting, deleting or inverting DNA segments of arbitrary length. As such, it mediates a variety of important cellular processes including chromosome segregation and viral infections. (See the review \cite{bio} for more details). 
Minimally, site-specific recombination requires both particular proteins (\textit{site-specific recombinases}) and two short (30--50bp) DNA segments (the \textit{crossover sites}) within one or two  DNA molecules (\textit{the substrate}). (More complex site-specific recombination systems may also require additional proteins (called \textit{accessory proteins}) and DNA sites (called \textit{enhancer sequences}).)  Site-specific recombinases can be broadly divided into two subfamilies: serine site-specific recombinases and tyrosine site-specific recombinases, based on their catalytic residues.
\\\\Site-specific recombination roughly has three stages (see Figure \ref{interstep}). First, two recombinase molecules bind to each of two crossover sites and bring them close together. (The sites together with the four bound recombinases is called the \textit{synaptic complex}.) Second, the crossover sites are cleaved, exchanged and resealed. (The precise nature of this intermediary step is determined by the  recombinase subfamily, see Assumption 3 and Figures \ref{serinesbio} and \ref{tyrosinesbio}.). And finally, the rearranged DNA (\textit{the product}) is released.
\\\\Multiple rounds of strand exchange can occur before releasing the DNA:  this process is known as \textit{processive recombination}.  (See Assumption 3 and Figure \ref{serinesmaths}.) This is in contrast to \textit{distributive recombination}, where multiple rounds of the entire process of recombination (including releasing and rebinding) occurs. Only serine recombinases can mediate processive site-specific recombination, but both types of recombinases can mediate distributive recombination.  In this work we use the term substrate to refer specifically to the DNA prior to the first cleavage. We treat processive recombination  as one extended process (with several intermediate exiting points for the reaction).

%

\subsection{Mathematical terminology}
\noindent A \textit{twist knot} is a knot that admits a projection with a row of $v\neq0$ vertical crossings and a \textit{hook}, as in Figure \ref{substrate} and is denoted by $C(\pm2,v)$. If $r=-2$, by flipping the top loop we get $r=+2$ and add a positive crossing to the row of $v$ crossings (see this isotopy illustrated in Figure \ref{substrateisotopy}). Thus from now on we assume that our substrate is the twist knot $C(2,v)$, $v\neq0$. (See \cite{Rolfsen, BZ, Cromwell, Kawauchi, Murasugi} for a detailed discussion on twist knots).
\\\\\textbf{Note:}  Twist knots can be generalized  to \textit{clasp knots}. A clasp knot $C(r,v)$ is a knot that has two non-adjacent rows of crossings, one with $r\neq0,\pm1$ crossings and the other with $v\neq0$ crossings (Figure \ref{claspknot}). (By adjacent rows of $r$ and $v$ crossings we mean that the two rows cannot be considered as a single row of $r + v$ crossings as they can in the case of the torus knots and links $T(2, r + v)$. A clasp knot $C(r, v)$ with $r = \pm2$ is a twist knot.)
\\\\We use the following terminology and notation. We consider the central axis  of the DNA double helix and therefore, when we illustrate DNA molecules we draw this axis (and not the two DNA backbones that make up the double helix). Let $J$ denote the twist knot substrate molecule. Once the synaptic complex has been formed, the \textit{recombinase complex}, $B$, denotes the convex hull of the four bound recombinase molecules together with the two crossover sites. (Note that $B$ is a topological ball). The \textit{recombinase-DNA complex, $J\cup B$, denotes the union of the substrate $J$ with the recombinase complex $B$}. Let $C=\mbox{cl}(\mathbb{R}^3-B)$ and let  $C\cap J$ denote the complement of the recombinase complex. If the recombinase complex meets the substrate in precisely the two crossover sites then we say the recombinase complex is a \textit{productive synapse}, see Figure \ref{productivesynapse}. In particular, for recombinases that utilize an enhancer sequence or accessory proteins,  the recombinase complex is a productive synapse if the accessory sites and proteins are sequestered from the crossover sites.\footnote{Note that if $B$ is a productive synapse then it can be thought of as a (2-string) tangle.  However, unlike in the traditional tangle model, the complement of $B$ may take a variety of forms (not necessarily that of a tangle), so we avoid this potentially confusing terminology.}

\subsection{Notation for families of knots and links that arise from site-specific recombination on a twist knot.}\label{familiesexplained}
\noindent We now discuss the three families of knots and links that we encounter in the main results of this paper. The families of knots and links illustrated in Figures \ref{family}, \ref{famsG1} and \ref{famsG2} are referred to as $F(p,q,r,s,t,u)$, $G1$ and $G2$ respectively.
\\\\ We note that $F(p,q,r,s,t,u)$ is a special family of knots and links. In \cite{Conwaytangle} it is shown that a standard rational tangle diagram corresponds to any expansion of a rational
number $\frac{p}{q}$ as a continued fraction; in this paper we choose the convention that a choice of the expansion in which all terms
have alternating sign gives an alternating diagram (see Figure \ref{crossings} for a convention on crossings and Figure \ref{montesinoss}b for an example of a rational tangle diagram).  A \textit{Montesinos link} is a link $L$ that admits a diagram $D$ composed of $m\geq3$ rational tangle diagrams $R_1,...,R_m$ and $k\geq0$ half twists glued together as in Figure \ref{montesinoss}a (and see e.g. \cite{LickThist}). Members of $F(p,q,r,s,t,u)$ are obtained by the numerator closure of Montesinos tangles of the form $(\frac{t}{tu+1},\frac{r}{rs+1},\frac{p}{pq+1})$. That is, for three standard rational tangle diagrams with fractions $\frac{t}{tu+1},\frac{r}{rs+1}$ and $\frac{p}{pq+1}$, take their partial sum as in Figure \ref{montesinoss}c and then the closure of the diagram as in Figure \ref{montesinoss}d. Denote the tangle with corresponding rational number $\frac{t}{tu+1}, \frac{r}{rs+1}$ and $\frac{p}{pq+1}$ by  $R_1,R_2$ and $R_3$ respectively. As in \cite{BFmaths}, we define the family of \textit{small Montesinos knots and links to be the family of Montesinos links for $i=3$, $R_i$ as above and $k=0$. Thus, our family of knots and links $F(p,q,r,s,t,u)$ illustrated in Figure \ref{family} is a subfamily of small Montesinos knots and links. }
\\\\ In the family $F(p,q,r,s,t,u)$ of knots and links, the variables $p,q,r,s,t,u$ describe the number of crossings between two strands in that particular row of crossings. Note that knots that are members of this family can be prime or composite and links belonging to this family can have up to three components. In this family, the variables $p, q, r, s, t, u$ can be positive, negative or zero. By letting the variables equal $0$ or $\pm1$ as appropriate, we obtain the subfamilies illustrated in Figure \ref{subfamilies}. Subfamily 1 is denoted by $F_{S_1}(0,q,r,s,t,u)$  with $|r|>0,|t|>1$. Subfamily 2 is denoted $F_{S_2}(\pm1,q,r,s,\pm1,u)$ with $|r|>1$. Subfamily 3 is denoted $F_{S_3}(\pm1,q,r,s,t,u)$ with $|r|, |t|>1$. Subfamily 4 is denoted $F_{S_4}(p,q,r,s,t,u)$ when we forbid $p,t,r=\{0,\pm1\}$. Subfamily 5 are composite knots or links $T(2, u)\sharp C(p,q)$ formed from a torus knot and a twist knot. Subfamily 6 is a subfamily of $F(p,q,r,s,t,u)$ with $p+q=0$. Subfamily 7 is a family of clasp knots and links, $C(r, s)$. (Recall that it is a generalization of the family of twist knots, which we consider in this paper as the substrate molecule for site-specific recombination.) Subfamily 8 is the family of torus knots and links, $T(2,r)$. Finally, subfamily 9 is the family of pretzel knots $K(p,s,u)$. (Note that some of the subfamilies in Figure \ref{subfamilies} are special cases of other subfamilies, for example in Subfamily 2, if we let $q=0$ then we get Subfamily 5. Similarly for Subfamilies 1 and 7, 3 and 6.)
\\\\In the families $G_1(k)$ and $G_2(k)$ of knots and links, the variable $k$ describes the number of crossings between the two strands. Depending on the value of $k$, we obtain either a knot or a link:
For $G_i(k)$ for $i=1,2$, if $k$ is odd, the members of these families are knots. If $k$ is even, then the members of this family are two component links. These families are illustrated in Figures \ref{famsG1} and \ref{famsG2}.
\\\\Note that there are a few knots and links that belong to both $F(p, q, r, s, t, u)$ and either $G_1(k)$ or $G_2(k)$. For example the trefoil knot has a projection as a member of $F(p,q,r,s,t,u)$ with $p=0,t,u=1,r=2,s=-1$, and a projection as a member of $G_2(k)$ with $k = 2$.

\section{Assumptions} \label{assumptions} \label{assumtions1}

\noindent In this section we state and explain the three assumptions about the recombinase complex, the substrate and the mechanisms of recombination. (Biological justifications for these assumptions can be found in \cite{BFbio,KDbio}).
\\\\We make the following three assumptions about the recombinase-DNA complex, which we state in both biological and mathematical terms. These assumptions are similar in \cite{BFmaths, BFbio}. However, for Assumption 2 in particular, we introduce new terminology and prove a necessary result in order to re-state this Biological Assumption in precise mathematical terms. In \cite{BFbio, KDbio} we provide experimental evidence showing that each of these assumptions is biologically reasonable. 

\begin{assumptionB} The synaptic complex is a productive synapse, and there is
a projection of the crossover sites which has at most one crossing between the sites and no
crossings within a single site.
\end{assumptionB}

\begin{assumptionM}
$B\cap J$ consists of two arcs and there is a projection of
$B\cap J$ which has at most one crossing between the two arcs, and no crossings within a single arc.
\end{assumptionM}

\noindent Fix a projection of $J$ such that $B\cap J$ has one of the forms illustrated in Figure \ref{ass1}. Observe that form $B1$ can be rotated by $90^{\circ}$ to obtain form $B2$. However, we list form $B1$ and $B2$ as two different forms to make subsequent figures easier to follow (similarly for forms $B3$ and $B4$). 
\\\\Note that hooked productive synapses, illustrated in Figure \ref{crossings}, are biologically possible because there exist many recombinases whose productive synapse is not characterized, and for these systems $B\cap J$ could be hooked. However, this does not contradict Assumption 1, since a hook has no projections with no crossings, but has projections where there is only one crossing. There is an isotopy of the substrate molecule taking a hook from a projection with two crossings to a projection with one crossing, without affecting the projection of the substrate molecule outside a neighbourhood of the hook (Figure \ref{neighbourhood of B inter J} illustrates this).

\begin{assumptionB} The synaptic complex does not pierce through a supercoil or a
branch point in a nontrivial way and the supercoiled segments are closely juxtaposed. Also,
no persistent knots or links are trapped in the branches of the DNA on the outside of the synaptic complex.
\end{assumptionB}

\noindent Here \textit{persistent knots or links}  are those that remain after a continuous deformation of the DNA molecule, keeping $B$ fixed. 

\noindent Before we can state Assumption 2 mathematically, we need to introduce some terminology.
\\\\We define a \textit{planar surface with twists} as in \cite{BFmaths}.  Consider a surface lying in a plane together with a finite number of arcs in the surface whose endpoints are on the boundary of the surface (see Figure \ref{plannarsurface}$(a)$). We can use this planar surface with arcs to obtain a non-planar surface by replacing a neighborhood of each arc in the original surface by a half-twisted band and removing the top and bottom ends of the band (see Figure \ref{plannarsurface}$(b)$). Figure \ref{plannarsurface} illustrates how such a surface can be obtained from a doubly-punctured planar disc together with a collection of arcs defining the twists. A \textit{planar surface with twists} is defined to be any surface which can be obtained from a planar surface in this way. 
\\\\Define a surface $D$ with boundary $J$ to be a  \textit{spanning surface for $J$} if $D$ is topologically equivalent to a \textit{doubly-punctured  planar disc with twists}  when $J$ is a twist knot (Figure \ref{plannarsurface}). (We can think of a spanning surface for $J$ as a soap film surface with boundary $J$.) In the construction of this spanning surface, in Figure \ref{plannarsurface}$(a)$, we choose the twisted band that replaces the arc connecting the boundary of the planar disc and the right-most puncture and the twisted band that replaces the arc connecting the two punctures, such that the corresponding crossings defined  on the non-planar surface with twists make a set of $+2$ horizontal crossings that we call a \textit{clasp}, illustrated in Figures \ref{plannarsurface} and \ref{spanningsurfacee}.
\\\\Figure \ref{pierce} shows examples of the relationship between the spanning surface $D$ and the recombinase complex. Observe that in illustrations (i) and (ii) $D\cap\partial B$ consists exactly two arcs. In illustration (iii), no matter how the spanning surface $D$ is chosen, $D\cap \partial B$ contains at least one circle as well as two arcs, whereas in illustration (iv) there is an isotopy that removes the circle in $D\cap \partial B$. Mathematically, a spanning surface $D$ is \textit{pierced non-trivially} by $B$ if and only if $D\cap\partial B$ contains at least one circle in addition to the required two arcs, and there is no ambient isotopy of $D$ that removes this additional circle.

\begin{claim} The intersection of any spanning surface for $J$ and $\partial B$ contains exactly two arcs.
\end{claim}
\label{claim}

\begin{proof} By Assumption 1, $B$ contains exactly two arcs of $J=\partial D$, thus $\partial B\cap J$ is precisely four points. It follows that the intersection of any spanning surface for $J$ with $\partial B$ contains exactly two arcs, whose endpoints are the four points $\partial B\cap J$. By Biological Assumption 2, $B$ does not pierce the interior of any spanning surface $D$ in a non-trivial way. Thus $D\cap\partial B$ consists of exactly two arcs and no circles that cannot be removed by an ambient isotopy of $D$. 
\end{proof}

\noindent Suppose that $D$ is a spanning surface for $J$. We know by Biological Assumption 2 that the supercoiled segments of the DNA molecule $J$ are closely juxtaposed, this means that we can visualize the spanning surface $D$ as a narrow soap film surface. In particular, this means that the {two arcs in $D\cap \partial B$ are each very short, so we can assume that they are co-planar. (Note that this does not mean that the crossover sites themselves ($\partial D\cap B$) are co-planar). 
\\\\We can now define a surface $D\cap C$ to be \textit{unknotted relative to $\partial B$} if there is an ambient isotopy of $C$, point-wise fixing $\partial B$, that takes $D\cap C$ to a doubly-punctured planar disc with twists, where the end points of the arcs defining the twists are disjoint from $\partial B$.  Illustration (ii) of Figure \ref{knotted v unknoted} shows an example of $D\cap C$ unknotted relative to $\partial B$. Illustration (i) shows a knot trapped in the substrate molecule outside of $B$.
\\\\We are now ready to state Assumption 2 mathematically.

\begin{assumptionM} $J$ has a spanning surface $D$ such that $D\cap\partial B$ consists of exactly two arcs, the two arcs are co-planar and $D\cap C$ is unknotted relative to $\partial B$.
\end{assumptionM}

\noindent The fact that $J$ has a spanning surface $D$ satisfying Assumption 2 means our model is independent of the projection of the substrate $J$, so we now fix a projection of $J$ as in Figure \ref{substrate} and from now on we work with this particular projection $J$. Note that here we are referring specifically to the substrate $J$ before the synaptic complex is formed. The conformations of the pre-recombinant recombinase-DNA complex are dealt with in Section \ref{lemmas}. 
\\\\Recall site-specific recombinases fall into two subfamilies, the serine recombinases and the tyrosine recombinases. The details of the mechanism differ depending on which subfamily the recombinase belongs to. Assumption 3 addresses the mechanism for each subfamily of recombinases.

\begin{Bass3S} Serine recombinases perform recombination via the 'subunit
exchange mechanism.' This mechanism involves making two simultaneous double-stranded
breaks in the sites, rotating two recombinase molecules in opposite sites by $180^{\circ}$ within the productive synapse and resealing the new DNA partners (Figure \ref{serinesbio}). In each subsequent round of processive recombination, the same set of subunits is exchanged and the sense of rotation remains constant.
\end{Bass3S}
\begin{Bass3T}After recombination mediated by a tyrosine recombinase, there is a projection of the crossover sites which has at most one crossing (Figure \ref{tyrosinesbio}).
\end{Bass3T}

\noindent The mathematical statement is as follows:

\begin{Mass3S} After (each round of processive) recombination mediated by a serine recombinase, there is precisely one additional crossing between the crossover sites. (see Figure \ref{serinesmaths}).
\end{Mass3S}
\begin{Mass3T}
After recombination mediated by a tyrosine recombinase, there is a projection of the crossover sites which has at most one crossing (Figure \ref{ass1}).
\end{Mass3T}

\section{Possible forms of the productive synapse and its complement}\label{lemmas}

\noindent In this section we determine the pre-recombinant and post-recombinant conformations of the recombination sites and all possible conformations of the DNA-protein complex.  We also prove the necessary background lemmas for Section \ref{theorems}.
\subsection{Possible forms of the productive synapse $\mathbf{B\cap J}$.}\label{assumtions2}
\noindent As a result of Assumption 2, we have fixed a projection of $J$ prior to cleavage such that $B\cap J$ has form $B1$, $B2$, $B3$ or $B4$, illustrated in Figure \ref{ass1}. It follows from Assumption 3  that after $n$ recombination events with serine recombinases, we have added a row of  either $n-1,n$ or $n+1$ identical crossings that can be positive, negative or zero. Without loss of generality, we assume that after $n$ recombination events with serine recombinases, we add a row of $n$ identical crossings that can be positive, negative or zero. Thus after $n$ recombination events our fixed projection of $B\cap J$ is isotopic to one of the forms $n1$ or $n2$ illustrated in Figure \ref{serinesmaths}. (Note that from $n1$ we can obtain $n2$ by rotating by $90^{\circ}$. However, we list them as separate forms in order to make it easier to follow the use of Figure \ref{serine products} in the proof of Theorem 2.)
\\\\For tyrosine recombinases, without loss of generality we assume that the post-recombinant projection of $B\cap J$ has one of the eight forms in Figure \ref{tyrosinesmaths}. Notice that conformations $B5,B6,B7$ and $B8$ are hooks. Hooks have no projections with no crossings but do have projections with one crossing, so we allow these conformations. Forms $B1, B3, B5$ and $B7$ are equivalent by a $90^{\circ}$ rotation, to forms $B2, B4, B6$ and $B8$ respectively. (We list them separately to make it easier to follow the use of Figure \ref{tyrosine products} in the proof of Theorem 1.) 

\subsection{Possible forms of the complement of the productive synapse $\mathbf{C\cap J}$}\label{prerecombinant characterization}
In this section we determine all the possible conformations of $C\cap J$, and determine the respective pre-recombinant conformations of  $B\cap J$ for each form of $C\cap J$. For simplicity, we will use the phrase `$C\cap J$ has a particular form' when we mean that `$C\cap J$ is ambient isotopic, pointwise fixing $\partial B$, to that form'. The forms of $C\cap J$ referred to in the lemma are illustrated in Figure \ref{all the possibilities for substrate}. 

\begin{lemma} Suppose that Assumptions $1$, $2$ and $3$ hold for a particular recombinase-DNA complex with substrate $J$. Let $J$ be a twist knot $C(2,v)$. Then $C\cap J$ has one of possible five forms listed below. For each of these, $B$ has corresponding forms: 
\\If $C\cap J$ has the form
\begin{itemize}
\item $C1$, then $B\cap J$ has the form $B1$
\item $C2$, then $B\cap J$ has the forms $B1, B3$ or $B4$
\item $C3$, then $B\cap J$ has the forms $B2, B3$ or $B4$
\item $C4$, then $B\cap J$ has the form $B1$
\item $C5$, then $B\cap J$ has the forms $B3$ or $B4$.
\end{itemize} 
\end{lemma}

\begin{proof}
By Assumption 2, we can choose a spanning surface $D$ to be a doubly-punctured planar disc with twists as in Figure \ref{plannarsurface}, such that $D\cap\partial B$ is two co-planar arcs and $D\cap C$ is unknotted rel $\partial B$. 
\\\\Consider the doubly-punctured disc that generates $D$ in Figure \ref{plannarsurface}$(a)$.  We can consider this punctured planar disc with arcs as a thrice-punctured $S^2$ with a collection of arcs connecting the punctures (the three punctures are numbered $1$, $2$ and $3$ in Figure \ref{plannarsurface}$(a)$). A thrice-punctured $S^2$ in $S^3$ with arcs connecting the three punctures can be regarded as  a graph with three points and a collection of arcs connecting them, as illustrated in Figure \ref{S2}.
\\\\We determine all possible conformations of $C\cap J$ in three steps as follows: 
\\\\First we consider all different possible locations of the specific sites $D\cap\partial B$ on the thrice-punctured $S^2$. The two specific
sites can be located either both on the boundary of one puncture of $S^2$ or on a combination of these, so we consider each case. Notice, however, that from the symmetry of the graph described above it is enough to consider only the cases where the sites are located either:
\begin{itemize}
\item \textit{Case (11):} both on the boundary of puncture 1, 
\item \textit{Case (22) or (33):} both on the boundary of  puncture 2 (equivalent to both on puncture 3),
\item \textit{Case (12) or (13):} one site on the boundary of  puncture 1 and the other on the boundary of puncture 2 (equivalent to one site on puncture 1 and the other on puncture 3) and
\item \textit{Case (23):} one site on the boundary of puncture 2 and the other on puncture 3.
\end{itemize}

\noindent Next, on the corresponding spanning surface $D$ (generated by the thrice-punctured $S^2$) we consider all possibilities for $D\cap B$, which can either be two discs or a (possibly twisted) band. 
\\\\Finally, we perform an appropriate isotopy of this spanning surface which maps it to a spanning surface having one of the conformations  of $C\cap J$ illustrated in Figure \ref{all the possibilities for substrate} as boundary. We do this for every case.
\\\\In Figures \ref{cases1}, \ref{cases3} and \ref{cases2} we illustrate the isotopy of $C\cap J$ to one of the standard forms $C1, C2, C3, C4$ or $C5$ (or we show that such a case is not allowed by assumption) for each case. Even though cases (22) and (33) (and (12) and (13)) are equivalent, we illustrate all of them since it may be more straightforward to visualise the isotopy in one case or the other.
\\\\In Figures \ref{cases1}, \ref{cases3} and \ref{cases2}, inside each box we have three sets of illustrations: 
\\\indent \textbf{Left illustration:} thrice punctured $S^2$ with (thin, long) arcs which define the twists on a non-planar surface  and (thick, short) arcs on the boundary of one of the punctures (or a combination of these punctures) defining the arcs $D\cap\partial B$. 
\\\indent \textbf{Middle illustration:} the spanning surface $D$ of our substrate $J=C(2,v)$ with the arcs $D\cap\partial B$ illustrated by a pair of thick, short arcs. 
\\\indent \textbf{Right illustration:} corresponding conformation of $C\cap J$. 
\\\\Since there are many cases and some of whose isotopies are not very complicated, we have illustrated all the cases in Figures \ref{cases1}, \ref{cases3} and \ref{cases2}. Here we describe two of the most involved in detail.
\\\\\textit{Case (11c):} Assume both arcs of $D\cap\partial B$ lie on the boundary of puncture 1 of the thrice-punctured $S^2\subset S^3$. Consider case (11c) as illustrated in Figure \ref{cases1}. The thrice-punctured $S^2$ generates the spanning surface $D$ illustrated. Figure \ref{deformation11c} illustrates a continuous deformation taking this conformation of $D$ to a conformation whose boundary is of the form C1.
\\\\\textit{Case (12a):} In Figure \ref{cases2} consider case (12a). The thrice-punctured $S^2$ generates the spanning surface $D$ illustrated. Figure \ref{deformation12a} illustrates a continuous deformation taking this conformation of $D$ to a conformation whose boundary is of the form C1.
\\\\Here $C\cap D$ is unknotted rel $\partial B$ so the left and middle forms illustrated in Figures \ref{cases1}, \ref{cases3} and \ref{cases2} yield (up to isotopy,
fixing $\partial B$) the corresponding forms of $C\cap J$ illustrated on the right images. Thus, we can also specify the pre-recombinant form of $B \cap J$ for each conformation of $C\cap J$ as shown in Figure \ref{all the possibilities for substrate}. 
\end{proof}

\noindent\textbf{Observations} 

\noindent Since $B\cap J$ contains at most one crossing, the component of $D$ with almost all of the twists of $C(2,s)$ must be contained in $C$. In form $C2$, while there may be twists to the right of $B$, they are
topologically insignificant, since they can be removed by rotating $C\cap D$ by some multiple of $\pi$. In form $C1$, any twists which had occurred above $B$ can be removed and added to the row of twists below $B$ by rotating $C\cap D$ by some multiple of $\pi$. These rotations can occur while pointwise fixing $B$. Thus the five forms of $C\cap J$ illustrated in Figure \ref{all the possibilities for substrate} are the only ones possible.

\section{Characterization of knots and links arising as products of site-specific recombination on a twist knot}\label{theorems}

\noindent In this section we prove Theorems 1 and 2 which determine all the putative DNA knot and link products of (non-distributive) site-specific recombination on a twist knot substrate. We show that these products belong to one of the three families of knots and links illustrated in Figure \ref{alternative family}. (These families of knots and links are defined in section 2.3). We also identify all knots and links that arise as products of distributive site-specific recombination. 
\\\\Here we use our preliminary work from Section \ref{familiesexplained} to prove our main results. In this section, we suppose that the substrate is a twist knot $C(2,v)$ and that all three of our assumptions hold for a particular recombinase-DNA complex. We prove Theorems 1 and 2 which characterize all possible knotted or linked products brought about by a non-distributive reaction with a tyrosine recombinase and a serine recombinase, respectively.
Most knotted and linked products are in the family $F(p, q, r, s, t, u)$. However, there are a series of products of site-specific recombination with a tyrosine recombinase that instead belong to one of $G_1(k)$ or $G_2(k)$ (see the proof of Theorem 1 and Figure \ref{tyrosine products}). In this section we also discuss knots that cannot arise as products of different scenarios of site-specific recombination on twist knots.

\begin{theoremm} Suppose that Assumptions 1, 2 and 3 hold for a particular \textbf{tyrosine} recombinase-DNA complex with substrate $J$. If $J$ is a twist knot $C(2,v)$ then the only possible products (of a non-distributive reaction) are the unknot, the Hopf link, $C(r,s)$ for $r=\{1,2,3,4\}$, $T(2,m)$, a connected sum $T(2,m)\sharp C(2,s)$, a member of the family $F(p,q,r,s,t,u)$ in Figure \ref{family} with $|r|\geq2,|t|=1$ or $2$,$|p|\leq 1$, or a member of the family of knot and links $G1$ or of the family of knots and links $G2$.
\end{theoremm}
\label{theoremtyrosines}

\noindent The possible products are illustrated in Figure \ref{tyrosine products}.

\begin{proof}
By Assumption 3, after recombination with a tyrosine
recombinase $B\cap J$ has one of the eight post-recombinant forms illustrated in Figure \ref{tyrosinesmaths}. By Lemma 1, $C\cap J$ has one of the five forms illustrated in Figure \ref{all the possibilities for substrate}. The products of recombination mediated by a tyrosine recombinase are obtained by replacing the pre-recombinant forms of $B\cap J$ in each of the forms of $C\cap J$ (in Figure \ref{all the possibilities for substrate}) with each of the eight post-recombinant forms of $B\cap J$ (in Figure \ref{tyrosinesmaths}). The resulting products are illustrated in Figure \ref{tyrosine products}.
\\\\More specifically, suppose that $J$ is $C(2,v)$. Then by Lemma 1, $C\cap J$ can have form $C1, C2, C3, C4$ or $C5$. Hence by Figure \ref{tyrosine products}, the possible products are the unknot, $C(r,s)$ for $r=\{1,2,3,4\}$, $T(2,m)$, a Hopf link, a connected sum $T(2,m)\sharp C(2,s)$, a member of the family $F(p,q,r,s,t,u)$ in Figure \ref{family} with $|t|=1$ or $2$,$|p|\leq 1$ and a knot or a link that has a projection in either $G1$ or $G2$. 
\\\\ Note that from Figure \ref{tyrosine products} we can see that all the possible products of site-specific recombination mediated by a tyrosine recombinase on a twist knot substrate belong to one of the subfamilies of $F(p,q,r,s,t,u)$ as illustrated in Figure \ref{subfamilies}, with one exception. For the image on column 7, row 5 of Figure \ref{tyrosine products}, depending on the value of $v$, we get different knots or links:
\\If $v$ is an odd number, then the product is a knot:
\\If $v$ is a negative odd number, the product is a knot belonging to family $G1$ with $k=|v|$.
\\If $v$ is a positive odd number, the product is a knot belonging to family $G2$ with $k=|v|-1$.
\\If $v$ is an even number, the product is a two component link:
\\If $v$ is a negative even number, the product is a link belonging to family $G1$ with $c=|v|$.
\\If $v$ is a positive even number, the product is a link belonging to family $G2$ with $c=|v|-1$.
\\\\Thus, with the exception of one sequence of products, all products of recombination with a tyrosine recombinase belong to the family of small Montesinos knots and links illustrated in Figure \ref{family}.
\end{proof}

\noindent  It follows from Theorem 1 that every product of recombination with tyrosine recombinases is a member of the family in Figure \ref{family} with $|t|=1,2$ and $|p|\leq1$, or a member of the families $G1$ and $G2$. Also, it follows from Figure \ref{subfamilies} that $C(2,s)$ (possibly with an additional trivial component) $T(2,m)$ and $T(2,m)\sharp C(2,s)$ can be obtained as knots or links in $F(p,q,r,s,t,u)$ with $|t|=1,2$ and $|p|\leq1$. 

\begin{theoremm} Suppose that Assumptions 1, 2 and 3 hold for a particular \textbf{serine} recombinase-DNA complex with substrate $J$. If $J$ is $C(2,v)$ then the only possible products (of a non-distributive reaction) are the $C(r,s)$, $T(2,m)$, a connected sum $T(2,m)\sharp C(2,s)$ and any member of the family in Figure \ref{family} with $|r|\geq2,t\neq0$ and $|p|\leq 1$.
\end{theoremm}
\label{theoremserines}

\noindent The possible products are illustrated in Figure \ref{serine products}.

\begin{proof}
By Assumption 3, after recombination with a serine
recombinase, $B\cap J$ has one of the two post-recombinant forms $n1$ and $n2$ illustrated in Figure \ref{serinesmaths}. Also, by Lemma 1, $C\cap J$ has one of the five forms illustrated in Figure \ref{all the possibilities for substrate}. For Assumption 3 for serine recombinases, for each of the forms of $C\cap J$, the products of recombination with serine recombinases are obtained by replacing each of the pre-recombinant forms of $B\cap J$ with their corresponding post-recombinant form of $B\cap J$ after $n$ rounds of processive recombination according to  Figure \ref{serinesmaths}. The resulting products are illustrated in Figure \ref{serine products}.
\\\\More specifically, suppose that $J$ is $C(2,v)$. Then according to Lemma 1, $C\cap J$ can have forms $C1, C2, C3, C4$ or $C5$. When $C\cap J$ has form $C1$, then $B\cap J$ must have form $B1$. It follows from Figure \ref{serinesmaths} that the post-recombinant form of $B\cap J$ must be of form $n2$. Thus, by replacing $B\cap J$ with $B1$ in $C1$, we obtain that the products can be any knot or link in subfamily 3 illustrated in Figure \ref{subfamilies}. When $C\cap J$ has form $C2$, then $B\cap J$ must have form $B2, B3$ or $B4$. In this case by Figure \ref{serinesmaths}, the post-recombinant form of $B\cap J$ must be of form $n1$ or $n2$. We see from form $C2$ in Figure \ref{all the possibilities for substrate} that the products can be any knots or links in subfamily 5 or subfamily 7 illustrated in Figure \ref{subfamilies}. A similar analysis is made on the other possible forms of $C\cap J$ to arrive to the conclusion that the products can be any knots or links in subfamilies $1,3,5,7$ or $8$ illustrated in Figure \ref{subfamilies} and thus, are members of the family $F(p,q,r,s,t,u)$ (See Figure \ref{serine products}).
\end{proof}

\noindent Table 1 summarizes the results of Theorems 1 and 2.
\\\\\textbf{Note:} Theorems 1 and 2 distinguish between the chirality of the product DNA molecules, since using our model we can work out the 
\underline{exact conformation of all possible} products of site-specific recombination starting with a particular twist knot susbtrate and site-specific recombinase. For example, starting with the twist knot substrate $C(2,-1)$ (a right-handed (or (+)) trefoil), according to our model, site-specific recombination mediated by a tyrosine recombinase yields $T(2,5)$, which is a (+) $5_1$ (among other products) and can never yield T(2,-5), which is a (-) $5_1$. For an explicit strategy see our paper \cite{KDbio}.

\subsection{Knots and links that cannot arise as products}

There are a number of simple knots and links that cannot arise as products of non-distributive site-specific recombination. 

\begin{corol} Suppose that Assumptions 1, 2 and 3 hold for a particular site-specific  recombinase-DNA complex with substrate a twist knot $C(2,v)$. Any product arising that falls outside of families $F(p,q,r,s,t,u)$, $G1$ or $G2$ must arise from distributive recombination.
\end{corol}

\noindent For example, the knot $8_{18}$ is a knot that is not Montesinos, thus it does not belong to our family of small Montesinos knots and links. It also does not belong to either $G1$ or $G2$, so $8_{18}$ is an example of a knot that cannot arise as a product of non-distributive recombination on a twist knot substrate.
\\\\\textit{Knots and links in $F(p,q,r,s,t,u)$ that cannot arise from recombination mediated neither by a serine recombinase, nor a tyrosine recombinase}. The knot $10_{141}$ cannot be expressed in $F(p,q,r,s,t,u)$ with $t\neq0$ and $|p|\leq1$. Recall that all products from recombination with a tyrosine recombinase or a serine recombinase belonging to $F(p,q,r,s,t)$ can be expressed with $r>2, t\neq0$ and $|p|\leq1$. Thus $10_{141}$ cannot arise as a product.
\\\\\textit{Knots that cannot arise from recombination mediated by a tyrosine recombinase}.
There are knots and links in $F(p,q,r,s,t,u)$ which do not have a projection with $|t|=\{\pm1,\pm2\}$ and $|p|\leq1$, for example, the knot $8_{11}=F(2, 2, 2,-1,-3, 0)$. By inspection we can  see that there is no way to express $8_{11}$ as a member of $F(p,q,r,s,t,u)$ with $t=\{\pm1, \pm2\}$ and $|p|\leq1$, hence $8_{11}$ is not a product of recombination with a tyrosine recombinase. The knot $10_{64}$ is another example of this.
\\\\\textit{Knots that can arise as products of recombination mediated by a serine recombinase, but not by a tyrosine recombinase}. In contrast with Theorem 1, any knot or link in the family illustrated in Figure \ref{family} with $t\neq 0$ (not just $t=\{\pm1, \pm2\}$) and $|p|\leq 1$, can occur as a consequence of Theorem 2. The knot $8_{11}$ mentioned above is an example of this; this knot is a possible product of recombination with a serine recombinase, but not with a tyrosine recombinase.

\section{Minimal crossing number of our model}\label{MCN}

\noindent In this section we prove Theorem 4 which shows that all the possible DNA knot and link products of site-specific recombination on a twist knot substrate are a very small fraction of all knots and links. We also further restrict the knot and link types of products that have minimal crossing number one more than that of the substrate.
\subsection{The growth of product knots and links is proportional to  $n^5$.}\label{MCN1}
To prove the main theorem of this section, Theorem 4, we will split the family $F(p,q,r,s,t,u)$ of knots and links into seven smaller subfamilies illustrated in Figure \ref{counting}. Theorem 4 is independent of how family $F(p,q,r,s,t,u)$ is split up since we are using these subfamilies to count all the possible knots and links belonging to $F(p,q,r,s,t,u)$.
\begin{definition} For a knot or link $K$ the \textit{minimal crossing number} MCN$(K)$ is the smallest number of crossings over all possible projections. For a knot or link $K$, denote its minimal crossing number by MCN$(K)$.
\end{definition}

\noindent The number of prime knots and links (links with up to two components and counting chiral pairs separately) with minimal crossing number $n$ grows exponentially as a function of $n$  \cite{CS}. By contrast, we now prove that the total number of knots and links with MCN$(K)=n$ that are putative products of site-specific recombination on a twist knot substrate grows linearly as a function of $n^5$. Our families include prime and composite knots and links with up to three components. For the purposes of this section, we do not distinguish handedness of chiral knots, however, even including both versions of chiral knots, still our family grows slower than the function of $n^5$ multiplied by 2. This actually means that all the possible prime knot and (two-component) link products of site-specific recombination on a twist knot substrate are a very small fraction of all knots and links.
\\\\First, we consider knots and links belonging to $F(p,q,r,t,s,u)$. Note that, while the knots and links in this family  have at most six non-adjacent rows containing $p,q,r,s,t$ and $u$ signed crossings respectively, it does not follow that the minimal crossing number of such a knot or
link is $|p|+|q|+|r|+|s|+|t|+|u|$. If the knot or link is not alternating, it is quite possible that the number of crossings can be significantly reduced. Thus, \textit{a priori}, there is no reason to believe that the number of knots and links in this product family should grow linearly with $n^5$.

\begin{definition} A link diagram is called \textit{reduced} if it does not contain any `removable' or `nugatory' crossings. A \textit{reduced alternating} link diagram is a link diagram that is reduced and also alternating.
\end{definition}

\noindent Murasugi \cite{Murasugi1} and Thistlethwaite \cite{Thist} proved
that any reduced alternating diagram has a minimal number of crossings.
Buck and Flapan \cite{BFmaths} used this to show that for a twist knot $C(r,s)$ if $r$ and $s$ have the same sign, then MCN($C(r, s)) = |r|+|s|-1$,
and if $r$ and $s$ have opposite sign then MCN($C(r, s))= |r| + |s|$. 
\\\\To prove our result, we consider a \textit{Hara-Yamamoto projection}: a projection of a knot or a link in which there is a row of at least two crossings and which has the property that if this row is cut off from the rest of the projection and the endpoints are resealed in the two natural ways, then both resulting projections are reduced alternating (see Figure \ref{HY}). Hara and Yamamoto showed that any Hara-Yamamoto projection has a minimum number of crossings \cite{Hara-Yamamoto}. 
\\\\We make use of the following theorem proved by Lickorish and Thistlethwaite, in \cite{LickThist}:

\begin{theorem}\textbf{(Lickorish-Thistlethwaite)} 
If a link $L$ admits an $n$-crossings projection of the form as in Figure \ref{montesinoss}(a) with $k=0$ and each $R_i$ a reduced alternating rational tangle diagram with one crossing between the two arcs at the bottom of each $R_i$ and at least one more crossing. Then $L$ cannot be projected with fewer than $n$ crossings.
\end{theorem}

\noindent We refer to such a projection as a \textit{reduced Montesinos diagram}. We can deduce from the theorem that any projection of a knot or link that is a reduced Montesinos diagram has a minimal number of crossings. \\\\We begin with two lemmas.

\begin{lemma} The number of distinct knots and links in the product family illustrated in Figure \ref{family} with MCN= $n$ grows linearly with $n^5$.
\end{lemma}

\begin{proof}
Fix $n$ and suppose $K$ is a knot or a link projection in the family of Figure \ref{family} with minimal crossing number $n$. Then this projection has $|p|+|q|+|r|+|s|+|t|+|u|$ crossings. We divide the proof into three cases:
\begin{case} \textit{$K$ reduced alternating or reduced Montesinos:} If projection $K$ is reduced alternating or a reduced Montesinos diagram, then $|p|+|q|+|r|+|s|+|t|+|u|=n$.
\end{case}

\noindent We now show that if $K$ is not reduced alternating or a reduced Montesinos diagram then it is ambient isotopic to one of 121 possible projections which have minimal number of crossings.

\begin{case} \textit{$K$ can be isotoped to a reduced alternating or reduced Montesinos diagram:} Figure \ref{strandmovement} illustrates an example of how to reduce the number of crossings in a projection $K$ that is not reduced alternating or reduced Montesinos. Observe that for the link in Figure \ref{strandmovement}, the part containing the rows of $r$ and $q$ crossings is alternating if and only if $r$ and $q$ have opposite signs. Similarly for the section containing the rows of $r$ and $u$ crossings. If $r$ and $q$ have the same sign, then by moving a single strand, this part of the knot or link becomes
alternating. This isotopy removes a crossing from both the $r$ row and the $q$ row and adds a single new crossing. Thus we reduce this part of the diagram from having $|r|+|q|$ crossings in a non-alternating form to having $(|r|-1)+(|q|-1)+1$ crossings in an alternating form. Similarly, for the middle and left hand side of the diagram, a non-alternating diagram having $|r|+|u|$ crossings is reduced to  an alternating for having $(|r|-1)+(|u|-1)+1$ crossings. So overall, our original non-alternating diagram having $|r|+|q|+|u|$ crossings is reduced by an isotopy of two strand movements to a reduced alternating diagram having $(|r|-1)+(|q|-1)+(|u|-1)+2=n$ crossings. Note that we can also change non-alternating diagrams to reduced Montesinos diagrams using strand movements like these. 
\end{case}

\begin{case} \textit{$K$ cannot be isotoped to either a reduced alternating or reduced Montesinos diagram:} There are also cases where we cannot obtain a reduced alternating or reduced Montesinos diagram via strand movements of $K$. We describe a specific example illustrated in Figure \ref{HY}.  Let $K$ be a knot or link diagram in our family $F(p,q,r,s,t,u)$ with $t,p=1,r>1,s=1,q,u<-1$. In its original form, the projection has $(r-1)+(|u|+1)+(|q|+1)$ crossings. The projection on the left of Figure \ref{HY} is Hara-Yamamoto because the projections (on the right) obtained by resealing the endpoints are both reduced alternating. Thus, this projection has a minimum number of crossings. 
\end{case}

\noindent We consider 121  cases according to the values of $p,q,r,s,t$ and $u$, and show that in all but the Hara-Yamamoto case $K$, the initial diagram is isotopic to a diagram that is either reduced alternating or reduced Montesinos and hence has minimal crossing number. Since there are so many cases, we display the results in Tables 2 to 5 rather than discussing each case individually. We make the following notes and observations with respect to the Tables.

\begin{itemize}
\item  To compute Tables 2, 3, 4 and 5, the family of knots and links $F(p,q,r,s,t,u)$ is broken down into seven smaller subfamilies, shown in Figure \ref{counting}. We count knots and links belonging to subfamilies denoted by  $F_{S_1}(0,q,r,s,t,u)$ with $|r|>0, |t|>1$, $F_{S_2}(\pm1,q,r,s,\pm1,u)$ with $|r|>1$, $F_{S_3}(\pm1,q,r,s,t,u)$ with $|r|, |t|>1$, $F_{S_4}(p,q,r,s,t,u)$ with $|t|,|r|,|p|>1$, $T(2,r)$, $K(q,s,u)$ and the $unlink$ (subfamilies illustrated with a double arrow in between indicate that they give the same knots and links, so we count only one of them). Observe that in subfamily $F_{S_2}(\pm1,q,r,s,\pm1,u)$, the rows of crossings containing $u$ and $q$ crossings are interchangeable, so we treat the variables $u$ and $q$ as interchangeable. Similarly, in subfamily $F_{S_3}(\pm1,q,r,s,t,u)$, the tangles $R_1$ and $R_2$ are interchangeable, so we treat the variables $r$ and $t$ as interchangeable and $s$ and $u$ as interchangeable. A similar consideration is given to subfamily $F_{S_1}(0,q,r,s,t,u)$ and $F_{S_4}(p,q,r,s,t,u)$. For certain specific values of $p,q,r,s,t$ and $u$, we may obtain a trivial knot or link. However, we do not specifically exclude these cases from our Tables.

\item \textit{For all tables:} \textit{column two} lists the form of the knot or link which has a minimal number of crossings (e.g. reduced alternating). If the knot or link is isotopic to a clasp, pretzel, or torus knot or link or a composition of any of these, we list the specific form. Also, if one of the knots or links contains a trivial component, we use the shorthand $+O$ to indicate this.  \textit{Column three} shows the number of strand movements needed to achieve a diagram with minimal number of crossings. We write an expression with ($\pm\sharp$?) at the end to indicate that there may be $\sharp$ more or less number of strand movements, depending on the values of the relevant variables. \textit{ Column four} shows the MCN of the corresponding reduced alternating or reduced Montesinos conformation. The MCN is listed as an unsimplified function of $p, q, r, s, t$ and $u$ to help the reader recreate the isotopy taking the original form to the minimal crossing form. As a consequence, in column four we write an expression with ($\pm\sharp$?) at the end to indicate that the MCN may be $\sharp$ smaller or bigger. For example, when the minimal crossing form of the knot is a clasp knot $C(r, s)$, if we do not know the signs of $r$ and $s$, on column two, we write an expression with $+1?$ and in column three we write an expression with $-1?$, see Figure \ref{strandmovement}. In \textit{column five} we obtain the upper bounds for the number of links in each case by expressing MCN$=n$ as a sum of nonnegative integers. This enables us to find an upper bound for the number of knots and links with MCN$=n$ in each case. Note that the upper bounds given are intended to be simple rather than as small as possible. In particular, a number of our cases overlap, and thus some knots and links are counted more than once.
\item \textit{For all tables:} We consider a knot or link and its mirror image to be of the same link type, and hence we do not count both. Thus without loss of generality, we assume that $r\geq0$. 
\end{itemize}

\noindent There are 118 nontrivial cases in Tables 2, 3, 4 and 5. Any knots and links appearing more than once in the tables are counted only once. Thus there are at most 111 distinct families of knots and links listed in the tables. The number of knots and link in each of these families is bounded above by $4n^5$ (in fact, for most of the cases there are significantly fewer than $4n^5$ knot and link types). It follows that for a given $n$, the number of distinct knots and links in the product family $F(p,q,r,s,t,u)$ which have MCN$=n$ is bounded above by $4n^5\times 111=444n^5$. In particular, the number of distinct knots and links with the form of Figure \ref{family} which have MCN$=n$ grows linearly with $n^5$.
\end{proof}

\noindent We now consider product knots and links belonging to $G1$ or $G2$. These come about as products of recombination with a tyrosine recombinase on a recombinase-DNA complex with conformation $C4$ illustrated in Figure \ref{all the possibilities for substrate} and the post-recombinant conformation $B=B6$. 

\begin{lemma} For a fixed $n$ there exists at most one knot type in $G1$ with MCN equal to $n$. Similarly, for $G2$.
\end{lemma}

\begin{proof} 
$G1$ is reduced alternating, and hence has minimal number of crossings. Thus, it is clear that the MCN$(G1)=4+|v|=n$. Similarly, $G2$ is reduced alternating and thus has minimal crossing number, so MCN$(G2)=3+v=n$. 
\end{proof}

\noindent We are now ready to prove Theorem 4.

\begin{theoremm} \label{thm4} The number of putative knots and links resulting from site-specific recombination on a substrate that is the twist knot $C(2,v)$ with MCN equal to $n$ grows linearly with $n^5$. 
\end{theoremm}

\begin{proof} There are at most $111+2=113$ non-trivial, distinct families of knots and links that are putative products of site-specific recombination on a substrate that is $C(2,v)$; $111$ belong to the family of small Montesinos knots and links illustrated in Figure \ref{family} (the number of such knots and links is bounded above by $4n^5$) and two belong to the families $G1$ and $G2$ illustrated in Figures \ref{famsG1} and \ref{famsG2}. It follows that for a given $n$, after recombination on a twist substrate, the number of distinct knots and links  which have MCN$=n$ is bounded above by $4n^5\times 113=452n^5$. In particular, the number of distinct knots and links that belong to the families $G1, G2$ and/or $F(p,q,r,s,t,u)$ that have MCN$=n$ grows linearly with $n^5$.
\end{proof}

\noindent It follows from Theorem \ref{thm4}  that the proportion of all knots and (two-component) links which are contained in the families $F(p,q,r,s,t,u)$, $G1$ and $G2$ decreases exponentially as $n$ increases. Thus, for a knotted or linked product, knowing its MCN and that it belongs to one of these families allows us to significantly narrow the possibilities for its precise knot or link type. The model described herein thus provides an important step in characterizing DNA knots and links which arise as products of site-specific recombination.

\subsection{Products whose MCN is one more than the substrate}\label{MCN2}
\noindent We now prove a more directly applicable theorem. Site-specific recombination often increases the MCN of a knotted or linked substrate by one, see for example Table 1 in \cite{BFbio}. If the substrate is $C(2,v)$, with minimal crossing number $m$ and the product of a single recombination event has MCN$=m+1$, then we can further restrict the resulting knot or link type. Recall the the MCN$(C(2,v))=2+|v|$ for $v<0$ and MCN$(C(2,v))=1+v$ for $v>0$. 
\\\\We remark that site-specific recombination that increases the minimal crossing number of the product by one could result in a change in the number of components. For example, if the substrate is $C(2,2)$ which is a one component link (a knot) one of the possible products according to Theorem 5 is $T(2,4)$, a two component link.

\begin{theoremm} \label{onemoreMCNtheorem}Suppose that Assumptions 1, 2, and 3 hold for a particular recombinase-DNA complex with substrate $J=C(2,v), v\neq 0$ and denote the MCN$(J)=n>0$. Let $L$ be the product of a single recombination event and suppose MCN$(L)=n+1$. Then: 
\\\\If $v>0$, $L$ is either: $C(2,v+1), C(2,-v), C(-2,v), C(-2,-(1+v)), C(3,v), T(2,\pm(2+v)), F_{S_1}(0,q,2,s,2,u)$ where $u+s=v$, $F_{S_2}(\pm1,\pm1,2,s,\pm1,u)$ where $u+s=v$ or $s\neq0$ or $F_{S_3}(\pm1,0,2,s,2,u)$ where $u+s=v$ and $s,u\neq0$. \\\\If $v<0$ $L$ is either: $C(2,2+|v|), C(2,-(1+|v|)), C(-2,1+|v|), C(-2,-(2+|v|)), C(3,v), C(-4,v)$, $T(2,\pm(3+|v|))$ or $F_{S_2}(\pm1,\pm1,2,s,\pm1,u)$ for $u+s=v$.
\end{theoremm}

\noindent Table \ref{theorem5table} summarises this information.

\begin{proof} Firstly, note that $n\geq 2$. Note also that if $v=1$ then $n=2$, but there are no nontrivial knots with minimal crossing number equal to 2, so the substrate must be the unknot, which is considered in  \cite{BFmaths} and \cite{BFbio}. We exclude the case when $v=1$.
\\\\For $n=3$, $C(2,v)$ is the trefoil knot $3_1$ (i.e., $v=-1$) so $L$ must be the Figure of eight knot $4_1=C(2,-2)$ or the torus link $T(2, \pm4)$, since these are the only knots and links with minimal crossing number equal to 4. 
\\\\Now assume  that $n\geq4$, that is $v\leq-2$ or $v\geq3$. By Assumption 1, there is a projection of $J$ such that $B\cap J$ has at most one crossing. Since $J = C(2,v)$, the proof of Lemma 1 shows that $C\cap J$ has the  forms $C1, C2, C3, C4$ or $C5$ (Figure \ref{all the possibilities for substrate}).
When $C\cap J$ has form $C1$, then $u+s=v$. By Assumption 3 and Figures \ref{serinesmaths} and \ref{tyrosinesmaths},  the post-recombinant form of $B\cap J$ is one of those illustrated in Figure \ref{tyrosinesmaths}. Thus any knotted or linked product $L$ has one of the forms illustrated in Figure \ref{tyrosine products}.  
\\\\Now, suppose that $L$ has one of the forms illustrated when $C\cap J$ has form $C2, C3, C4$ or $C5$. $L$ cannot be either $T(2,2)\sharp C(2,v)$ or $T(2,-2)\sharp C(2,v)$ because MCN$(L)=n+2$. $L$ can certainly not be a Hopf link, an unknot or $C(2,v)$ with a trivial component, since MCN$(L)\neq n+1$. Finally, $L$ cannot be $C(4,v)$ because if $v>0$, MCN$(L)=3+v$ and of $v<0$, MCN$(L)=4+|v|$.
\\\\If $L=C(2,n)$ then $n=1+v$ or $-v$ when $v>0$ or $n=2+|v|$ or $v-1$ when $v<0$. If $L=C(k,v)$ then $k=3\  \forall v$, $k=-2$ for $v>0$ and $k=-4$ for $v<0$. If $L=T(2,n)$ then for $v>0$ $n=\pm(v+2)$ and for $v<0$ $n=\pm(3+|v|)$. 
\\\\If $L=F_{S_1}(0,q,2,s,t,u)$ for $s+u=v$ for some value of $t$ then $L$ has a projection in this product subfamily with $t=\pm2$. So we can assume $L$ has a projection of the form $F_{S_1}(0,q,2,s,\pm2,u)$ with $u+s=v$. If $t=-2$, for $v<0$ MCN$(L)=1+|u|+|s|+1=3+|v|=n+1$ so this case is possible and for $v>0$ MCN$(L)=|-2|+u+1+s=3+v\neq n+1$ so this case in not possible. If $t=+2$, for $v<0$ MCN$(L)=2+|u|+2+|s|=4+|v|\neq n+1$ so this case is not possible and for $v>0$ MCN$(L)=1+u+1+s=2+v=n+1$ so this case is possible. So $L=F_{S_1}(0,q,2,s,-2,u)$ only for $v<0$ and $L=F(2,s,2,u)$ for $v>0$.
\\\\Now suppose that $L$ has one of the forms illustrated when $C\cap J$ has form $C1$. Suppose $L$ has a projection of the form $F_{S_2}(\pm1,q,2,s,\pm1,u)$ with $u+s=v$. For $v>0$ and $s=0$, $q$ must be $0$, however, this is isotopic to $T(2,u)$, which has MCN$=v$, thus this is not allowed. For $v>0$ and $s\neq0$, $q=\pm1$ and for $v<0$, $q=\pm1$. That is, for $v>0$, $L=F_{S_2}(\pm1,\pm1,2,s,\pm1,u)$ for $s\neq0$ and $s+u=v$, and for $v<0$ $L=F_{S_2}(\pm1,\pm1,2,s,\pm1,u)$ for $u+s=v$.
\\\\If $L$ has a projection of the form $F_{S_3}(\pm1,q,2,s,t,u)$ for some value of $t$, then $L$ has a projection in this product subfamily with $t=\pm2$. Thus we now assume that $L$ has a projection of the form $F_{S_3}(\pm1,q,2,s,\pm2,u)$ with $u+s=v$. If $t=2$, for $v<0$, MCN$(L)=2+|u|+2+|s|+|q|=4+|v|+|q|>n+1$ for any value of $q$, so this case is not possible, for $v>0$ MCN$(L)=1+u+1+s+|q|=2+v+|q|$, so this case is possible for $q=0$. In this particular case, if one of $u$ or $s$ equals 0, then MCN$(L)=3+v+|q|>n+1$ for any value of $q$, so $L=F(q,2,s,2,u)$ only for $v>0$ and $s,u\neq0$. If $t=-2$, for $v>0$ MCN$(L)=|-2|+u+1+s+|q|=3+v+|q|>n+1$ for any value of $q$, so this case is not allowed and for $v<0$ MCN$(L)=|-1|+|u|+2+|s|+|q|=3+|v|+|q|$ so this case is possible for $q=0$. In this particular case, if one of $u$ or $s$ equals 0, then MCN$(L)=4+|v|+|q|>n+1$ for any value of $q$, so $L=F(q,2,s,-2,u)$ only for $v<0$ and $s,u\neq0$. In summary, $L=F_{S_3}(\pm1,0,2,s,-2,u)$ for $v<0$ and $L=F(0,2,s,2,u)$ for $v>0$, both with $u,s\neq0$ are the only possibilities allowed.
\\\\Finally, suppose $L$ has a projection of the form $G1$. Then $v<0$ and MCN$(L)=4+|v|>3+|v|=n+1$, thus $L$ cannot have this conformation.
Suppose has a projection of the form $G2$, then $v>0$ and  MCN$(L)=3+v>2+v=m+1$ and so $L$ cannot have this projection either. This completes the proof.
\end{proof}

\section{Applications of our model}\label{applications}

\noindent

\noindent We discuss how the model developed here can be a useful tool to analyse previously uncharacterized data in a variety of setttings in \cite{KDbio}. \\\\These applications fall into four broad categories:  \textit{Application 1:} our model can help determine the order of products of processive recombination. \textit{Application 2:} in the common situations where the products of site-specific recombination have MCN one more than the MCN of the substrate, our model can help reduce the number of possibilities for these products. \textit{Application 3:} our model can help predict products of processive and distributive recombination. \textit{Application 4:}  our model can help distinguish between products of processive and distributive recombination. 
\\\\To give a flavour of how to use this model, we conclude with a simple example of Applications 1 and 2.
\\\\\textit{Application 1.} Our model can be used to help understand processive recombination mediated by a serine recombinase. Using Figure \ref{serine products} and Table 1 which summarize the conclusions of Theorem 2, we can narrow the possibilities for the sequence of products in multiple rounds of processive recombination. Suppose that for a twist knot substrate of the form $C(-2,v)$ with $v\neq 0$, experimental conditions minimize distributive recombination and the products of multiple rounds of processive recombination are twist knots, unknots (or $C(-2,s)+O$) and the connected sum of a torus knot and a twist knot $C(-2,s)\sharp T(2,m)$. Then from  Figure \ref{serine products}$(g)$ we can determine that recombination happens from the twist knot substrate to the clap knot with a trivial component $C(-2,v)+O$, product of the first round of recombination, to the connected sum of torus knots and clasp knots, product of the second round of recombination. Moreover, any products of further rounds of recombinations are connected sums of the form $C(-2,v)\sharp T(2,m)$ with increasing minimal crossing number.
\\\\\textit{Application 2.} We now demonstrate an application of Theorem 5. Suppose the twist knots $C(2,5)$ and $C(2,7)$ (which have MCN equal to 6 and 8 respectively) are used as substrates for a site-specific recombination reaction with a tyrosine recombinase, where experimental conditions eliminate distributive recombination  and products are knots and links with minimal crossing number 7 and 9. In this case the minimal crossing number is not sufficient to determine the knot type, since there are 7 knots, 8 two-component links and 1 three-component link with MCN=7 and 49 knots, 61 two-component links and 22 three-component links with MCN=9. However, we can use Theorem 6 and Table 7 to significantly reduce the number of possibilities for these products. It follow from Theorem 6 that the possible seven-crossing products are $7_1$, $7_2$, $7_3$, $7_6$, $7^2_2$, , $7_3^2$, or $3_1\sharp4_1$; and the possible nine-crossing products are $9_1$, $9_2$, $9_3$, $9_8$, $9_{11}$, $9^2_1$, , $9^2_{10}$, $6_1\sharp 3_1$, or $4_1\sharp 5_2$. In Table \ref{thm7app} we show how to do this. We have reduced from 16 choices for 7-noded knots to just 7, from 132 possibilities for 9-noded knots and links to just nine possibilities. Thus, Theorem 6 can help to significantly reduce the knot and link type of products of site-specific recombination that add one crossing to the substrate.

\section{Acknowledgements}
We wish to thank Ken Baker, Erica Flapan, Julian Gibbons, Mauro Mauricio and Ken Millett for insightful discussions, as well as the referees for their careful reading and helpful comments. DB is supported in part by EPSRC Grants EP/H0313671, EP/G0395851 and EP/J1075308, and thanks the LMS for their Scheme 2 Grant.  KV is supported by EP/G0395851.



\clearpage
\clearpage
\begin{center}
\begin{table}
\begin{center}
\begin{tabular}[h]{ll p{10cm}}\hline
\textbf{Recombinase type} & \textbf{Substrate} & \textbf{Product}\\\hline
Tyrosine & $C(2,v)$ & unknot, $C(r,s)$ for $r=1,2,3,4$, $T(2,m)$, Hopf link, $T(2,m)\sharp C(2,s)$, $F(p,q,r,s,t,u)$ with $|t|=1$ or $2$,$|p|\leq 1$, knots or links in families $G1$ or $G2$\\
Serine & $C(2,v)$ & $C(r,s)$, $T(2,m)$, $T(2,m)\sharp C(2,s)$, $F(p,q,r,s,t,u)$ with $|p|\leq 1$ and $t\neq0$\\\hline
\end{tabular}
\end{center}
\caption{\footnotesize{Products of non-distributive recombination predicted by our model.}}
\end{table}
\end{center}

\begin{center}
\begin{table}[h]{\footnotesize
\begin{center}
\begin{tabular}{l l l l p{2.1cm}}
\hline
Values of $p,q,r,s,t,u$ for $r\geq 0$& 
Minimal crossing form & 
Strands moved & 
MCN as a sum of non-negative integers & 
Upper bound on number of links\\\hline
$p=0,t,r\geq2, u,s\leq-1$&Reduced Montesinos&0&$t+|u|+r+|s|$&$n^3$\\
$p=0,t,r\geq2,u,s\geq1$&Reduced alternating&2&$(t-1)+(u-1)+(r-1)+(s-1)+2$&$n^3$\\
$p=0,t\leq-2, r\geq2,u,s\geq1$&Reduced alternating&1&$(r-1)+(s-1)+|t|+u+1$&$2n^3$\\
$p=0,s\geq1,u\leq-1, \mbox{ wlog }|u|<|s|$&Reduced alternating&$|u|\ \ \ +1?$&$|t|+r+(s-|u|)\ \ \ -1?$&$4n^2$\\
$p=0,s=-u$&$T(2,t+r)$&$|s|$&$|t+r|$&$1$\\
$p=0,s,u=0$&$T(2,t+r)$&$0$&$|t+r|$&$1$\\
$p=0,r=1,t=\pm1$&$T(2,(u\pm1)+(s+1))$&$0$&$|(u\pm1)+(s+1)|$&$1$\\
$p=0,r=1$&$C(t,u+s+1)$&1?&$|t|+|u+s+1|\ \ \ -1?$&4n\\
$p=0,u=0,t,r\geq2,s\geq1$&Reduced alternating&$2$&$(t-1)+(r-1)+(s-2)+2$&$n^2$\\
$p=0,u=0,t\leq-2,r\geq2,s\leq-1$&Reduced alternating&$1$&$(|t|-1)+r+(|s|-1)+1$&$n^2$\\
$p=0,u=0,t\leq-2,r\geq2,s\geq1$&Reduced alternating&$1$&$|t|+(r-1)+(s-1)+1$&$n^2$\\
$p=0,u=0,t,r\geq2,s\leq-1$&Reduced alternating&$0$&$t+r+|s|$&$n^2$\\\hline
$p=0,t=0$&$T(2,r)$&$0$&$|t|$&$1$\\
$p,r,t=0$&unlink&$0$&$0$&$0$\\
$p,r,t=\pm1$&$K(u\pm1,s\pm1,q\pm1)$&$0$&$(|u\pm1|)+(|s\pm1|)+(|q\pm1|)$&$8n^2$\\\hline
\end{tabular}
\end{center}
\caption{Theorem 4: The minimal crossing forms of knots and links in subfamilies $F_{S_1}(0,q,r,s,t,u), 5,6$ and $7$ illustrated in Figure \ref{counting}.}}
\end{table}
\end{center}

\begin{center}
\begin{table}[h]{\footnotesize
\begin{center}
\begin{tabular}{p{4.3cm} l l l p{2.3cm}}
\hline
Values of $p,q,r,s,t,u$ for $r\geq 0$& 
Minimal crossing form & 
Strands moved & 
MCN as a sum of non-negative integers & 
Upper bound on number of links\\\hline
$p,t=\pm1,u,q=0$&$C(r,s)+O$&$1?$&$|r|+|s|\ \ \ -1?$&$4n$\\
$p,t=\pm1,r=1$&$K(u,s+1,q)$&0&$|u|+|s+1|+|q|$&$4n^2$\\
$p,t=\pm1,r=1,s=1$&$T(2,u)\sharp T(2,q)$&$0$&$|u|+|q|$&$2n$\\
$p,t=\pm1,r>1,q=0$&$T(2,u)\sharp C(r,s)$&$1?$&$|u|+r+|s|\ \ \ -1?$&$4n^2$\\
$p,t=\pm1,r>1,uq=-1$&$T(2,r)$&$0$&$r$&$1$\\
$p,t=\pm1,r > 1, uq = 1, s = 0$& $C(\pm2, r)$& $0$ &$2+r\ \ \ -1?$& $2$\\
$p,t=\pm1,r > 1, u\geq1, q = 1, s > 0$& Reduced alternating &1& $u + (r - 1) + (s - 1) + 2$ $n^2$\\
$p,t=\pm1,r > 1, u = q = 1, s < 0$& Reduced alternating &1& $r + (-s - 1) + 2$& $n$\\
$p,t=\pm1,r > 1, u\leq-1, q = -1, s > 0$& Reduced alternating &$2$& $-u + (r - 1) + (s - 2) + 2$ $n^2$\\
$p,t=\pm1,r > 1, u, q < 0, s\leq0$& Reduced alternating &0& $-u - q + r -s$& $n^3$\\
$p,t=\pm1,r > 1, u, q > 1, s = 0$& Reduced alternating &$2$ &$(u - 1) + (q - 1) + (r -2) + 2$ $n^2$\\
$p,t=\pm1,r > 1, u < -1, q > 1, s = 0$& Reduced alternating& 1&$-u + (q - 1) + (r - 1) + 1$& $n^2$\\
$p,t=\pm1,r > 1, |u| > 1, |q| = 1, s = 0$ &$C(r \pm 1, u)$& 0 &$-u + (r \pm 1)\ \ \ -1?$& $4n$\\
$p,t=\pm1,r > 1, qs = -1$& $T (2, r + u \pm 1)$& 1 &$|r + u \pm 1|$ &1\\
$p,t=\pm1,r > 1, u > 0, q = 1, s < 0$ &Reduced alternating& 1& $u + r + (-s - 1) + 1$& $n^2$\\
$p,t=\pm1,r>1,u\leq-2,q=1,s\leq-2$&Reduced alternating&1&$(-u-2)+r+(-s -2) +1$&$ n^2$\\
$p,t=\pm1,r > 1, u, q > 0, s = 1$ &Reduced alternating& 1& $u + q + (r - 1) + 1$ &$n^2$\\
$p,t=\pm1,r > 1, u < -1, q > 0, s = 1$ &Reduced alternating &1& $(-u - 1) + q + (r - 1)$& $n^2$\\
$p,t=\pm1,r > 1, u < -1, q = 1, s > 1$& Reduced alternating &2& $(-u - 1) + (r - 1) + (s - 1) + 2$&$ n^2$\\
$p,t=\pm1,r > 1, u > 1, q = -1, s < 0$& Reduced alternating& 1 &$(u - 1) + r - s + 1$& $n^2$\\
$p,t=\pm1,r > 1, u > 1, q = -1, s = 2 $&Trivial &2& $0\neq n$&$ 0$\\
$p,t=\pm1,r > 1, p > 1, q = -1, s > 2$& Reduced alternating &3 &$(u - 2) + (r - 1) + (s - 3) + 2 $&$n^2$\\
$p,t=\pm1,r > 1, |u|, |q| > 1, s < 0$& Reduced Montesinos& 0 &$|u| + |q| + r - s$&$ 4n^3$\\
$p,t=\pm1,r > 1, |u|, |q| > 1, s > 1$& Reduced Montesinos &1&$ |u| + |q| + (r - 1) + (s - 1) + 1$&$ 4n^3$\\
$p,t=\pm1,r > 1, u < -1, q = -2, s = 1$& $K(u,r-1,2)$ &1& $-u + 2 + (r - 1)$& $n$\\
$p,t=\pm1,r > 1, u, q < -2, s = 1$& Hara–Yamamoto &1&$-u + (-q - 1) + (r - 1)$ &$n^2$\\\hline
\end{tabular}
\end{center}
\caption{Proof of Theorem 4: The minimal crossing forms of knots and links in subfamily $F_{S_2}(\pm1,q,r,s,\pm1,u)$ illustrated in Figure  \ref{counting}.}}
\end{table}
\end{center}

\begin{center}
\begin{table}[h]{\footnotesize
\begin{center}
\begin{tabular}{p{4.5cm} l l p{4.5cm} p{2.5cm}}
\hline
Values of $p,q,r,s,t,u$ for $r\geq 0$& 
Minimal crossing form & 
Strands moved & 
MCN as a sum of non-negative integers & 
Upper bound on number of links\\\hline
$p=1,t,r\geq2,q=-1,-1\leq u,s\leq0, \mbox{ not both }u,s=0$&Reduced alternating&$0$&$|r|+|s|+|t|+|u|+(|q|(\pm1))$&$n^4$\\
$p=1,t,r\geq2,u,s,q\leq-1$&Reduced Montesinos&$0$&$|r|+|s|+|t|+|u|+(|q|(\pm1))$&$n^4$\\
$p=1,t,r\geq2,u,s,q=1$&Reduced alternating&$2$&$(r-1)+(t-1)+3$&$n$\\
$p=1,t,r\geq2,u=0,s,q=1$&Reduced alternating&$3$&$(r-2)+(t-3)+1$&$2n$\\
$p=1,t,r\geq2,u=0,s,q>1$&Reduced alternating&$1$&$(r-1)+(s-1)+t+1$&$2n^2$\\
$p=1,t,r\geq2,u,s,q>1$&Reduced alternating&$2$&$(r-1)+(t-1)+(s-1)+(u-1)+2$&$n^4$\\
$p=1,t,r\geq2,-1\leq u,s\leq0,q=1\mbox{ not both }u,s=0$&Reduced alternating&$2$&$(r-1)+(t-1)+1$&$n$\\
$p=1,t,r\geq2,u,s\leq-2,q\geq2$&Reduced Montesinos&$0$&$r+t+|u|+|s|+q$&$n^4$\\
$p=1,t\leq-2,r\geq2,0\leq u\leq1,-1\leq s\leq0,q=-1$&Reduced alternating&$1$&$r+|t|+u+|s|$&$4n^3$\\
$p=1,t,s,q\leq-2,r,u\geq2$&Reduced Montesinos&$0$&$r+|t|+u+|s|+|q|$&$4n^4$\\
$p=1,r\geq2,t\leq-2,-1\leq u,s\leq0,q=1$&Reduced Montesinos&$1$&$r+|t|+(|u|-1)+|s|+1$&$2n^3$\\
$p=1,r\geq2,t\leq-2, u,s\leq-2,q\geq2$&Reduced Montesinos&$1$&$r+(|t|-1)+(|u|-1)+|s|+q+1$&$2n^4$\\
$p=1,t,r\geq2,s=0 u=0,q=-1$&Reduced alternating&$1$&$(r-1)+(t-1)+1$&$2n$\\
$p=1,t,r\geq2,u=1,s,q=-1$&Reduced alternating&$2$&$(r-1)+(t-2)+1$&$2n$\\
$p=1,t,r,u\geq2,q,s\leq-2$&Reduced Montesinos&$1$&$r+(t-1)+(u-1)+|s|+|q|$&$2n^4$\\
$p=1,r\geq2,t\leq-2,u=0,s=1,q=-1$&Reduced alternating&$1$&$|t|+r$&$2n$\\
$p=1,r\geq2,t\leq-2,u,s=1,q=-1$&Reduced alternating&$2$&$|t|+r-1$&$n$\\
$p=1,r,u,s\geq2,t,q\leq-2$&Reduced Montesinos&$1$&$(r-1)+(s-1)+|t|+u+|q|+1$&$n^4$\\
$p=1,t,r\geq2,u=0,q=1,s=-1$&Reduced alternating&$0$&$t+r+2$&$2n$\\
$p=1,t,r\geq2,u,q=1,s=-1$&Reduced alternating&$2$&$(t-1)+r+2$&$2n$\\
$p=1,t,r,u,q\geq2,s\leq-1$&Reduced Montesinos&$1$&$(t-1)+(u-1)+r+|s|+q+1$&$2n^4$\\
$p=1,r\geq2,t\leq-2,u=0,0\leq s\leq1,q=1$&Reduced alternating&$0$&$|t|+r+2$&$2n$\\
$p=1,r\geq2,t\leq-2,0\leq s\leq1,u,q=1$&Reduced alternating&$1$&$(|t|-1)+r+3$&$2n$\\
$p=1,r\geq2,t\leq-2,u,s,q\leq-1$&Reduced Montesinos&$2$&$(|t|-1)+(|u|-1)+|s|+|q|+r+1$&$2n^4$\\
$p=1,t,r\geq2,u=1,s=0,q=-1$&Reduced alternating&$1$&$(t-1)+r$&$2n$\\
$p=1,t,r\geq2,u,s=1,q=-1$&Reduced alternating&$2$&$(t-1)+(r-1)+1$&$n$\\
$p=1,t,r,s,u\geq2,q\leq-1$&Reduced Montesinos&$2$&$(t-1)+(r-1)+(u-1)+(s-1)+|q|+1$&$n^4$\\
$p=1,r\geq2,t\leq-2,0\leq u\leq1,s=-1,q=1$&Reduced alternating&$2$&$(t-2)+(r-1)+1$&$2n$\\
$p=1,r,u\geq2,t\leq-2,q\geq1,s\leq-1$&Reduced Montesinos&$0$&$r+u+q+|t|+|s|$&$2n^4$\\
$p=1,r=s=1$&$T(2,q)\sharp C(t,u)$&$1?$&$(|q|+1)+(|t|\ \ \ -1?)+(|u|\ \ \ -1?) \ \  \ +1?$&$4n^2$\\
$p=1,q=0$&$C(r,s)\sharp C(t,u)$&$2?$&$(|r|\ \ \ -1?)+(|s|\ \ \ -1?)+(|t|\ \ \ -1?)+(|u|\ \ \ -1?)+1\ \ \ +2?$&$8n^3$\\
$p=1,q=0,us=-1$&$T(2,r)\sharp T(2,t)$&$2$&$|t|+|r|+1$&$2n$\\
$p=1,q,s=0$&$T(2,r)\sharp C(t,u)$&$1?$&$(|t|\ \ \ -1?)+(|u|\ \ \ -1?)+|r|+1\ \ \ +1?$&$4n^2$\\
$p=1,u,s=0$&$C(t+r,q)$&$1?$&$(|t+r|\ \ \ -1?)+((|q|+1)-1?)\ \ \ +1?$&$4n$\\
$p=1,u,s,q=0$&$T(2,t+r)$&$0$&$|t+r|+1$&$1$\\
$p=1,t,r=0$&unknot&$|q|$&$0$&$0$\\\hline
\end{tabular}
\end{center}
\caption{Proof of Theorem 4: The minimal crossing forms of knots and links in subfamily $F_{S_3}(\pm1,q,r,s,t,u)$ illustrated in Figure  \ref{counting}.}}
\end{table}
\end{center}

\begin{center}
\begin{table}[h]{\footnotesize
\begin{center}
\begin{tabular}{p{4.5cm} l l p{4.5cm} p{2.5cm}}
\hline
Values of $p,q,r,s,t,u$ for $r\geq 0$& 
Minimal crossing form & 
Strands moved & 
MCN as a sum of non-negative integers & 
Upper bound on number of links\\\hline
$r,t,p\geq2,-1\leq u,s,q\leq0, \mbox{no two of }u,s,q=0$&Reduced alternating&$0$&$t+|u|+r+|s|+p+|q|$&$2n^2$\\
$r,t,p\geq2,u,s,q\leq-1$&Reduced Montesinos&$0$&$t+|u|+r+|s|+p+|q|$&$2n^5$\\
$r,t,p\geq2,0\leq u,s,q\leq1, \mbox{no two of }u,s,q=0$&Reduced alternating&$3$&$(t-1)+(|u|-1)+(r-1)+(|s|-1)+(p-1)+(|q|-1)+3$&$2n^2$\\
$r,t,p\geq2,u,s,q\geq2$&Reduced Montesinos&$3$&$(t-1)+(|u|-1)+(r-1)+(|s|-1)+(p-1)+(|q|-1)+3$&$2n^5$\\
$r\geq2,t,p\leq-2,u,q>1,s<-1$&Reduced Montesinos&$0$&$|t|+|p|+|s|+r+u+q$&$2n^5$\\
$r\geq2,t,p\leq-2,0\leq u,q\leq1,-1\leq s\leq0\mbox{ no two of }u,s,q=0$&Reduced alternating&$0$&$|t|+|p|+|s|+r+u+q$&$2n^2$\\
$t,r\geq2,p\leq-2,u,s<-1,q>1$&Reduced Montesinos&$0$&$|u|+|p|+|s|+r+t+q$&$4n^5$\\
$t,r\geq2,p\leq-2,-1\leq u,s\leq0,0\leq q\leq1,\mbox{ no two of }u,s,q=0$&Reduced alternating&$0$&$|u|+|p|+|s|+r+t+q$&$4n^2$\\
$t,r\geq2,p\leq-2,u=0,s,q\geq1$&Reduced alternating&$1$&$t+(r-1)+(s-1)+|p|+3+q$&$4n^5$\\
$t,r\geq2,p\leq-2,u,s,q\geq1$&Reduced alternating&$2$&$(t-1)+(u-1)+(r-1)+(s-1)+|p|+2+q$&$4n^5$\\
$r\geq2,t,p\leq-2,u,q\leq-1,0\leq s\leq1$&Reduced alternating&$2$&$(|t|-1)+(|u|-1)+(|p|-1)+(|q|-1)+|s|+2+r$&$2n^5$\\
$r\geq2,t,p\leq-2,u=0,s,q\leq-1$&Reduced alternating&$1$&$|t|+|u|+(|p|-1)+(|q|-1)+|s|+2+r$&$2n^5$\\
$r\geq2,t,p\leq-2,s,q\geq1,0\leq u\leq1$&Reduced alternating&$1$&$(r-1)+(s-1)+|p|+q+|t|+u+1$&$2n^5$\\
$r\geq2,t,p\leq-2,u,q\geq1,s=0$&Reduced alternating&$0$&$r+|p|+q+|t|+u$&$2n^5$\\
$r\geq2,t,p\leq-2,q\leq1,-1\leq s,u\leq0$&Reduced alternating&$1$&$(|p|-1)+(|q|-1)+|s|+r+|s|+t+1$&$4n^5$\\
$p,r,t\geq2,u,s\geq0,q\leq0, \mbox{no two of }u,s,q=0$&Reduced Montesinos&$2$&$(t-1)+(u-1)+(r-1)+(s-1)+p+|q|+2$&$3n^5$\\
$p,r,t\geq2,u\geq0,q,s\leq0, \mbox{no two of }u,s,q=0$&Reduced Montesinos&$1$&$(t-1)+(u-1)+|r|+|s|+p+|q|+1$&$3n^5$\\
$p,r\geq2,t\leq-2,u,s\geq0,q\leq0, \mbox{no two of }u,s,q=0$&Reduced Montesinos&$1$&$|t|+u+(r-1)+(s-1)+p+|q|+1$&$2n^5$\\
$p,r\geq2,t\leq-2,u,q\geq0,s\leq0, \mbox{no two of }u,s,q=0$&Reduced Montesinos&$1$&$|t|+u+(p-1)+(q-1)+r+|s|+1$&$2n^5$\\
$t,r\geq2,p\leq-2,u,s\geq0,q\leq0, \mbox{no two of }u,s,q=0$&Reduced Montesinos&$3$&$(|t|-1)+(|u|-1)+(p-1)+(q-1)+(r-1)+(s-1)+3$&$2n^5$\\
$t,r\geq2,p\leq-2,s\geq0,u,q\leq0, \mbox{no two of }u,s,q=0$&Reduced Montesinos&$2$&$t+u+(|p|-1)+(|q|-1)+(r-1)+(s-1)+2$&$2n^5$\\
$t,r\geq2,p\leq-2,u\geq0,s,q\leq0, \mbox{no two of }u,s,q=0$&Reduced Montesinos&$2$&$r+s+(|p|-1)+(|q|-1)+(t-1)+(u-1)+2$&$2n^5$\\
$r\geq2,t,p\leq-2,u,s\geq0,q\leq0, \mbox{no two of }u,s,q=0$&Reduced Montesinos&$2$&$|t|+u+(|p|-1)+(|q|-1)+(r-1)+(s-1)+2$&$2n^5$\\
$p,r\geq2,t\leq-2,s,u,q\leq0, \mbox{no two of }u,s,q=0$&Reduced Montesinos&$1$&$p+|q|+(|t|-1)+(|u|-1)+r+|s|+1$&$2n^5$\\
$r\geq2,t,p\leq-2,u\geq0,s,q\leq0, \mbox{no two of }u,s,q=0$&Reduced Montesinos&$1$&$|t|+u+(|p|-1)+(|q|-1)+r+|s|+1$&$2n^5$\\
$r\geq2,t,p\leq-2,s\geq0,u,q\leq0, \mbox{no two of }u,s,q=0$&Reduced Montesinos&$3$&$(|t|-1)+(|u|-1)+(|p|-1)+(|q|-1)+(r-1)+(s-1)+3$&$n^5$\\
$r\geq2,t,p\leq-2,s,u,q\leq0, \mbox{no two of }u,s,q=0$&Reduced Montesinos&$2$&$(|t|-1)+(|u|-1)+(|p|-1)+(|q|-1)+r+s+2$&$n^5$\\
$r=1,s=-1$&$C(t,u)\sharp C(p,q)$&$2?$&$(|t|\ \ \ -1?)+(|u|\ \ \ -1?)+(|p|\ \ \ -1?)+(|q|\ \ \ -1?)\ \ \ +2?$&$8n^2$\\
$r=1,s=-1,p=\pm1$&$T(2,q)\sharp C(t,u)$&$1?$&$(|t|\ \ \ -1?)+(|u|\ \ \ -1?)+(|q|\pm1)\ \ \ (+1?)$&$8n^2$\\
$r=1,s=-1,p=\pm1,q=\mp2$&$T(2,u)\sharp T(2,q)$&$0$&$(|u|\pm1)+(|q|\pm1)$&$2n$\\\hline
\end{tabular}
\end{center}
\caption{Proof of Theorem 4: The minimal crossing forms of knots and links in subfamily $F_{S_4}(p,q,r,s,t,u)$ illustrated in Figure  \ref{counting}.}}
\end{table}
\end{center}

\begin{center}
\begin{table}[h]
\begin{center}
\begin{tabular}[h]{lcl}
When &$L=$& for\\\hline 
$v>0$&$C(2,n)$& $n=1+v$ or $-v$\\
$v<0$&        & $n=2+|v|$ or $-(|v|+1)$\\
$v>0$&$C(-2,n)$&$n=v$ or $-(1+v)$\\
$v<0$&        & $n=1+|v|$ or $-(|v|+2)$\\
$\forall v$& $C(k,v)$& $k=3$\\
$v>0$&       & $k=-2$\\
$v<0$&       & $k=-4$\\
$v>0$& $T(2,n)$& $n=\pm(2+v)$\\
$v<0$&         & $n=\pm(3+|v|)$\\
$v>0$& $F_{S_1}(0,q,2,s,2,u)$&  $u+s=v$\\
$v>0$&$F_{S_2}(\pm1,q,2,s,\pm1,u)$ &$u+s=v, s\neq0, q=\pm1$\\
$v<0$&$F_{S_2}(\pm1,q,2,s,\pm1,u)$ &$u+s=v, q=\pm1$\\
$v>0$&$F_{S_3}(\pm1,q,2,s,2,u)$ &$u+s=v, s,u\neq0, q=0$\\\hline
\end{tabular}
\end{center}
\caption{Summary of Theorem 5.}
\label{theorem5table}
\end{table}
\end{center}

\begin{center}
\begin{table*}
\begin{center}
\begin{tabular}[b]{|l|l|}
\hline
\textbf{Products with 7 crossings} & \textbf{Products with 9 crossings}\\\hline
$C(2,6)=7_2$* & $C(2,8)=9_2$*\\
$C(2,-5)=7_2$* & $C(2,-7)=9_2$*\\
$C(-2,5)=7_2$* & $C(-2, 7)=9_2$*\\
$C(-2,-6)=7_2$*& $C(-2,-8)=9_2$*\\
$C(3,5)=5_1^2$ & $C(3,7)=9^2_1$*\\
$T(2,\pm7)=7_1$* & $T(2,\pm9)=9_1$*\\
$F_{S_1}(0,q,2,1,2,4)=7^2_3$* & $F_{S_1}(0,q,2,1,2,6)=9^2_{10}$*\\
$F_{S_1}(0,q,2,2,2,3)=7^2_3$* & $F_{S_1}(0,q,2,2,2,5)=9^2_{10}$*\\
$F_{S_2}(\pm1,1,2,1,\pm1,4)=7_3$* & $F_{S_1}(0,q,2,3,2,4)=9^2_{10}$*\\
$F_{S_2}(\pm1,-1,2,1,\pm1,4)=5_1$ & $F_{S_2}(\pm1,1,2,1,\pm1,6)=9_3$*\\
$F_{S_2}(\pm1,1,2,2,\pm1,3)=7_2^2$* & $F_{S_2}(\pm1,-1,2,1,\pm1,6)=7_1$ \\
$F_{S_2}(\pm1,-1,2,2,\pm1,3)=$unlink &$F_{S_2}(\pm1,1,2,2,\pm1,5)=7_2^2$\\
$F_{S_2}(\pm1,1,2,3,\pm1,2)=7_6$*& $F_{S_2}(\pm1,-1,2,2,\pm1,5)=$Hopf link\\
$F_{S_2}(\pm1,-1,2,3,\pm1,2)=$unknot & $F_{S_2}(\pm1,1,2,3,\pm1,4)=9_{11}$*\\
$F_{S_2}(\pm1,1,2,4,\pm1,1)=7^2_3$* & $F_{S_2}(\pm1,-1,2,3,\pm1,4)=5_2$\\
$F_{S_2}(\pm1,-1,2,4,\pm1,1)=$Hopf link & $F_{S_2}(\pm1,1,2,4,\pm1,3)=7_2^7$\\
&$F_{S_2}(\pm1,-1,2,4,\pm1,3)=5_1^2$\\
&$F_{S_2}(\pm1,1,2,5,\pm1,2)=9_8$*\\
&$F_{S_2}(\pm1,-1,2,5,\pm1,2)=4_1$\\
&$F_{S_2}(\pm1,1,2,6,\pm1,1)=9_{10}^2$*\\
&$F_{S_2}(\pm1,-1,2,6,\pm1,1)=$Hopf link\\
&$F_{S_3}(\pm1,0,2,1,2,6)=7_2$\\
$F_{S_3}(\pm1,0,2,1,2,4)=5_2$ & $F_{S_3}(\pm1,0,2,2,2,5)=6_1\sharp 3_1$*\\
$F_{S_3}(\pm1,0,2,2,2,3)=3_1\sharp4_1$*& $F_{S_3}(\pm1,0,2,3,2,4)=4_1\sharp 5_2$*\\\hline
\end{tabular}
\end{center}
\caption{\small{Example of a possible application to Theorem \ref{onemoreMCNtheorem}. Given recombination mediated by a tyrosine recombinase on the substrates $C(2,5)$ (MCN $= 6$) and $C(2,7)$ (MCN $= 8$)  where experimental conditions eliminate distributive recombination, we use Table 6 to list all the possible 7 and 9 noded products of this reaction.  Only the products that are isotopic to a knot and link with MCN one more than the substrate are possible products of this reaction and we denote these with a star (*). }}
\label{thm7app}
\end{table*}
\end{center}

\clearpage
\clearpage
\begin{figure}
\begin{center}
\includegraphics[width=10cm]{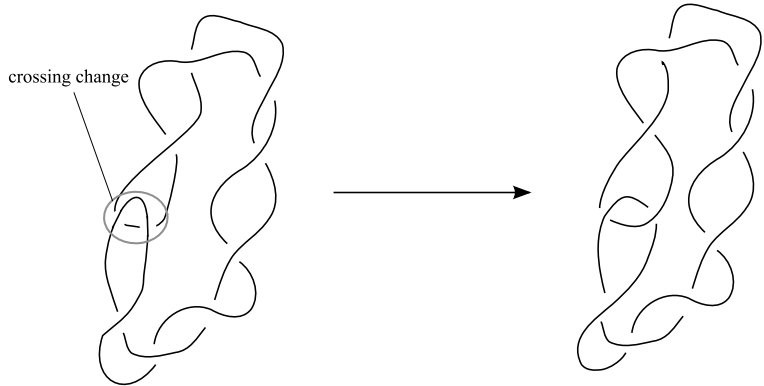}
\caption{\footnotesize{Twist knots are ubiquitous DNA knots. In the cell all DNA is supercoiled (like an over-used phone cord) so an unknot can be transformed to a twist knot by a single crossing change.}}
\label{supercoiledtotwist}
\end{center}
\end{figure}

\begin{figure}
  \centering
   \subfloat[][]{\label{family}\includegraphics[width=.35\textwidth]{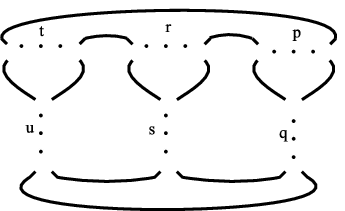}}
\hspace{0.5cm} 
  \subfloat[][]{\label{famsG1}\includegraphics[width=0.25\textwidth]{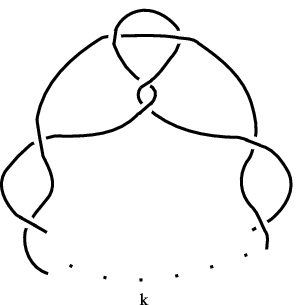}}
\hspace{0.5cm} 
  \subfloat[][]{\label{famsG2}\includegraphics[width=0.25\textwidth]{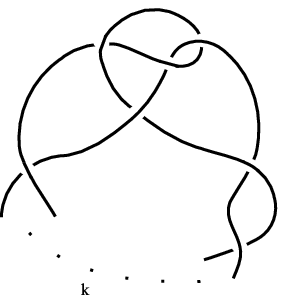}}
\caption{\footnotesize{All possible knots and links resulting from recombination on a twist knot must fall into one of these three families: (a) The family $F(p,q,r,s,t,u)$ of knots and links. Most knotted and linked products are in this family. (b) The family $G1$ of knots and links. (c) The family $G2$ of knots and links. For $K\in G1$ or $G2$, $k$ odd $\Rightarrow \ K$ is a knot and $k$ even $\Rightarrow$ $K$ is a two component link. }}
\label{alternative family}
\end{figure}

\begin{figure}
\begin{center}
\includegraphics[width=13cm]{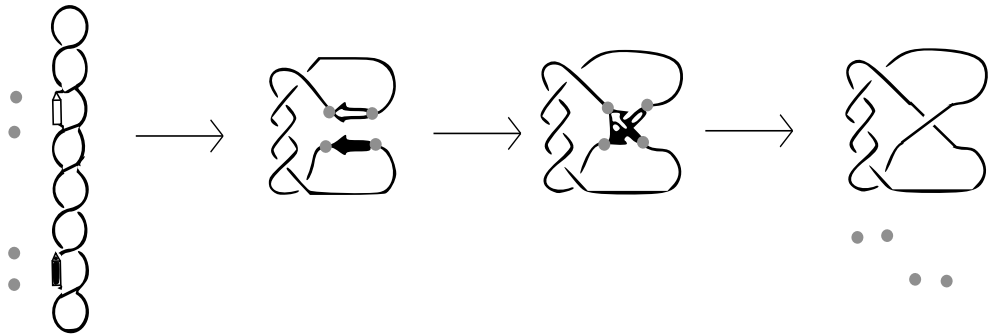}
\caption{\footnotesize{The line represents the axis of the double helix of the substrate DNA molecule. The recombinase dimers (grey circles) bind at each of the two specific sites (filled and hollow arrows) and the sites are brought together forming the synaptic complex with crossover sites juxtaposed (second image from the left). After cleaving, exchanging and resealing the DNA, the proteins dissociate completing the reaction.}}
\label{interstep}
\end{center}
\end{figure}

\begin{figure}
  \centering
    \subfloat[][]{\label{serinesbio}\includegraphics[width=1.1\textwidth]{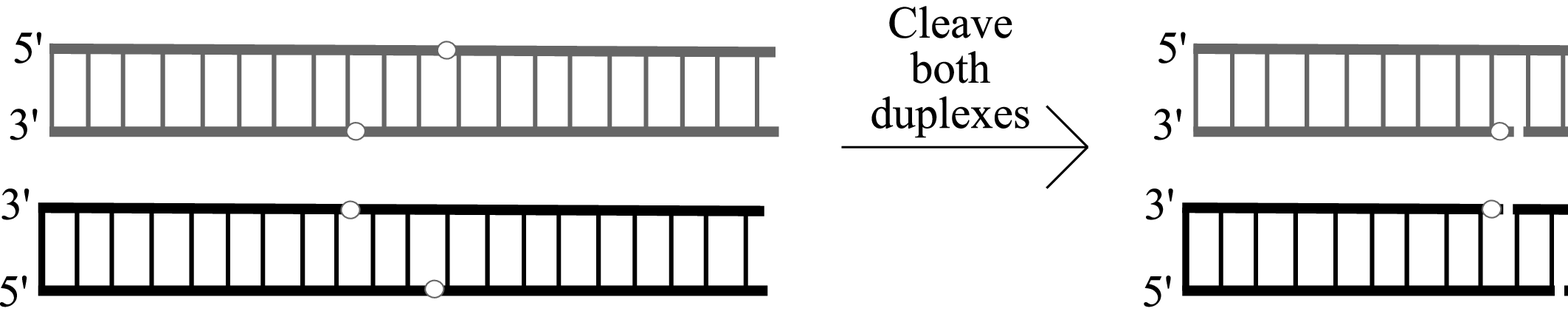}}
  \vspace{1cm}
  \subfloat[][]{\label{tyrosinesbio}\includegraphics[width=.8\textwidth]{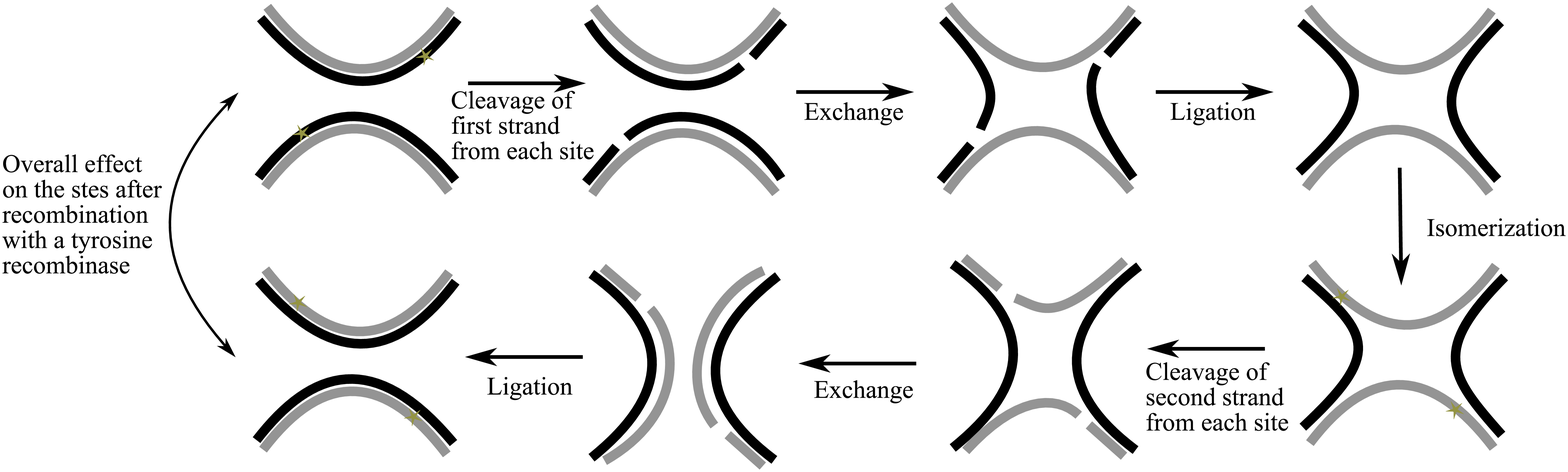}}
\caption{\footnotesize{(a). Biological Assumption 3: Serine recombinase. Serine recombinases perform simultaneous double-stranded breaks, rotate one half of the recombinase complex relative to the other by $180^{\circ}$ and rebind the DNA (b). Biological Assumption 3: Tyrosine recombinase. Tyrosine recombinases cleave one strand from each duplex, exchange the cleaved strands, and ligates them to form a Holliday junction (rightmost two panels). Isomerization of this junction alternates the catalytic activity and the same process happens with the other two DNA strands. These images are modifications of Figures 3 and 11 in \cite{bio}.}}
\label{bio}
\end{figure}

\begin{figure}
\begin{center}
\includegraphics[width=18cm]{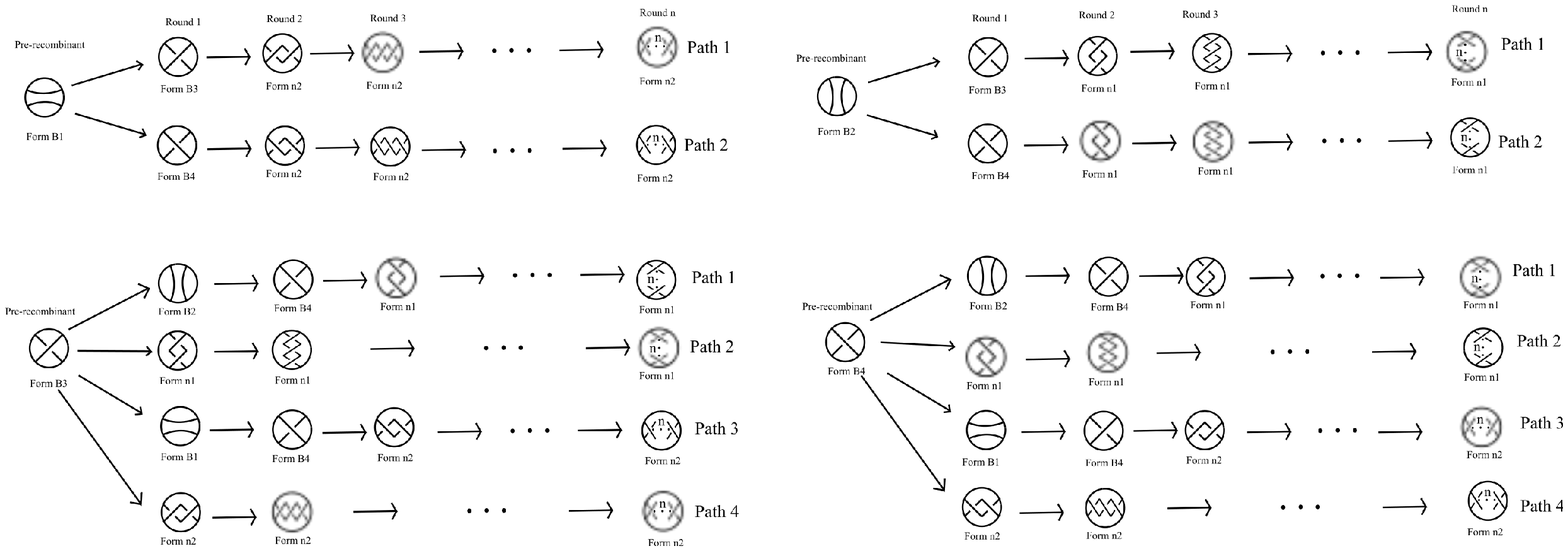} 
\caption{\footnotesize{Mathematical Assumption 3: Serine recombinases. Begin with all possible  projections of the pre-recombinant conformation of the recombinase complex, with zero or one crossings. Follow with projections of the post-recombinant conformations of the productive synapse at each round of processive recombination.}}
\label{serinesmaths}
\end{center}
\end{figure}

\begin{figure}
  \centering
  \subfloat[][$C(r,v)$]{\label{claspknot}\includegraphics[width=.15\textwidth]{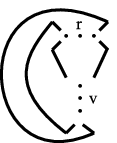}}
  \hspace{0.5cm}
  \subfloat[][C(2,v)]{\label{substrate}\includegraphics[width=.17\textwidth]{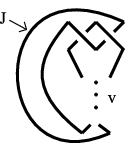}}
  \hspace{0.5cm}
  \subfloat[][Isotopy from $C(-2,v-1)$ to $C(2,v)$]{\label{substrateisotopy}\includegraphics[width=.4\textwidth]{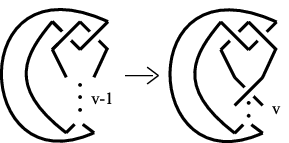}}
    \hspace{0.5cm}
  \subfloat[][]{\label{crossings}\includegraphics[width=.5\textwidth]{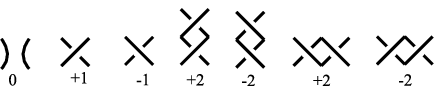}}
\caption{\footnotesize{Background terminology. (a) The clasp knot C(r,v) with two nonadjacent rows of crossings, one with $r \neq 0,±1$ crossings and the other with $v\neq 0$ crossings. (b) The substrate we consider here and in \cite{KDbio}, the twist knot $C(2, v)$. Note $r$ is now a \textit{hook} of 2 crossings. (c) A continuous deformation taking the twist knot $C(-2, v)$ to the twist knot $C(+2, v + 1)$. (d) Crossing sign convention used in this paper.}}
\label{defs}
\end{figure}

\begin{figure}
\begin{center}
\includegraphics[width=15cm]{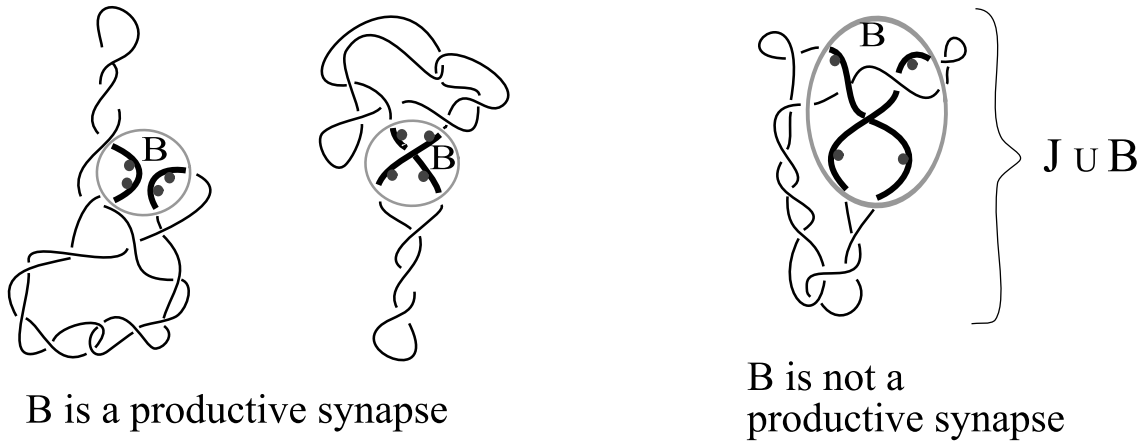} 
\caption{\footnotesize{Productive synapse. The thin black lines illustrate the central axis of the DNA molecule. We assume that the recombinase complex is a productive synapse.  $B$ (light grey circle) denotes the smallest convex region containing the four bound recombinase molecules (small grey discs) and the two crossover sites (highlighted in black). \textit{Left and middle:} $B$ is a productive synapse. \textit{Right:} $B$ is not a productive synapse. In this case we cannot draw $B$ such that only the two crossover sites are inside it without also including the third (horizontal, thin) strand.}}
\label{productivesynapse}
\end{center}
\end{figure}

\begin{figure}
\begin{center}
\includegraphics[width=10cm]{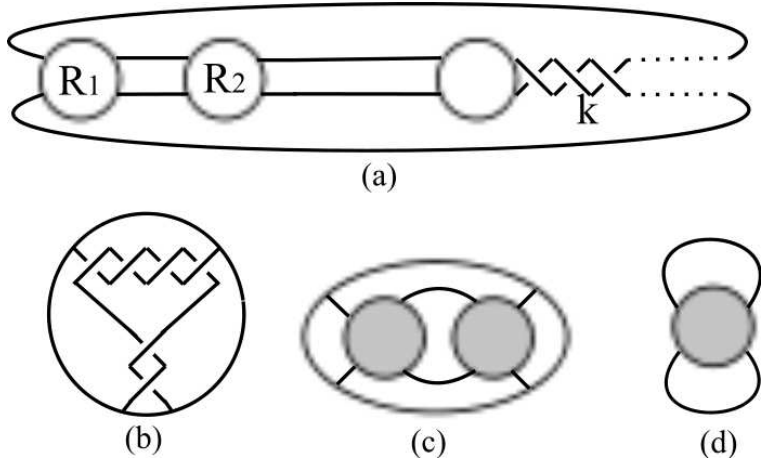} 
\caption{\footnotesize{(a). A Montesinos knot or link has a projection as illustrated here. The $R_i$ are rational tangles. (b) A rational tangle with alternating crossings. (c) The \textit{partial sum of two tangles}, refers to the tangle diagram resulting from the insertion of two tangle diagrams into the shaded discs. (d) Numerator closure of a tangle.}}
\label{montesinoss}
\end{center}
\end{figure}

\begin{figure}
\begin{center}
\includegraphics[width=18.4cm]{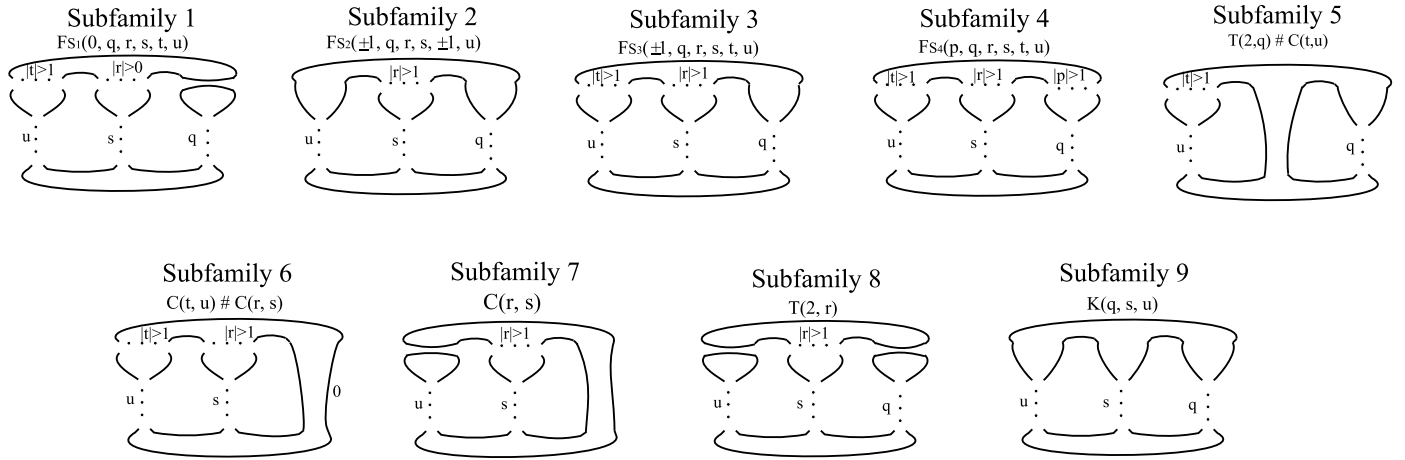} 
\caption{\footnotesize{Subfamilies of the family illustrated in Figure \ref{family}. The nine subfamilies obtained by setting $p, q, r,s,t$, and/or $u$ equal to $0$ or $\pm1$ in the family of knots and links $F(p,q,r,s,t,u)$. \textit{Top:} product subfamily $F_{S_1}(0,q,r,s,t,u)$ with $|r|,|t|>1$, product subfamily $F_{S_2}(\pm1,q,r,s,\pm1,u)$ with $|r|>1$, product subfamily $F_{S_3}(\pm1,q,r,s,t,u)$ with $|r|, |t|>1$, product family $F_{S_4}(p,q,r,s,t,u)$ with $|t|,|r|,|p|>1$, product subfamily of composite knots $T(2, u)\sharp C(p,q)$. \textit{Bottom:} product subfamily $F(-1,1,r,s,t,u)$ with $|t|,|r|>1$, product subfamily of clasp knots and links $C(r, s)$, product subfamily of torus knots and links $T(2,r)$, product subfamily of pretzel knots $K(p,s,u)$.}}
\label{subfamilies}
\end{center}
\end{figure}

\begin{figure}
\begin{center}
\includegraphics[width=8cm]{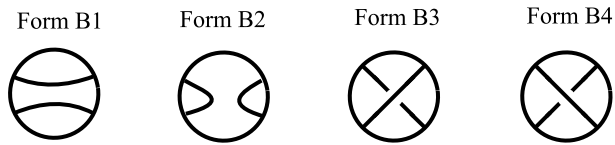} 
\caption{\footnotesize{Assumption 1:  Projections of the pre-recombinant productive synapse. Assumption 1 states that there is a projection of the pre-recombinant productive synapse with at most one crossing. Note that it does allow productive synapses like the hook, where there is a projection
with one crossing but no projections with zero crossings. Assumption 3 for tyrosine recombinases: After recombination with a tyrosine recombinase, the productive synapse has a projections with at most one crossing. }}
\label{ass1}
\end{center}
\end{figure}

\begin{figure}
\begin{center}
\includegraphics[width=7cm]{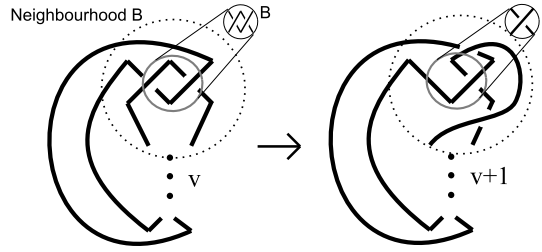} 
\caption{\footnotesize{By continuously deforming $B\cap J$ within a neighbourhood of the hook, we can obtain a projection of the hook with exactly one crossing. This affects the projection of the rest of the substrate only by adding one positive crossing to the row of $v$ crossings.}}
\label{neighbourhood of B inter J}
\end{center}
\end{figure}

\begin{figure}
\begin{center}
\includegraphics[width=7cm]{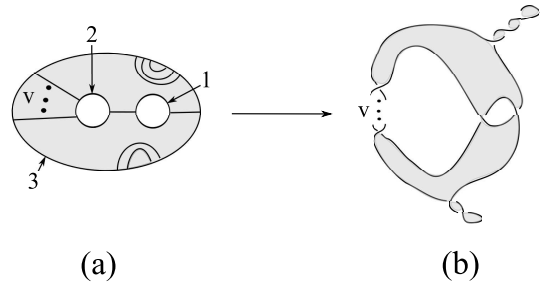}
\caption{\footnotesize{Obtain a \textit{planar surface with twists} by replacing a neighborhood of each arc by a half-twisted band. Here our planar surface is a doubly punctured disc and our non-planar surface with twists is a surface whose boundary is the twist knot $C(2,v)$.}}
\label{plannarsurface} 
\end{center}
\end{figure}

\begin{figure}
\begin{center}
\includegraphics[width=7cm]{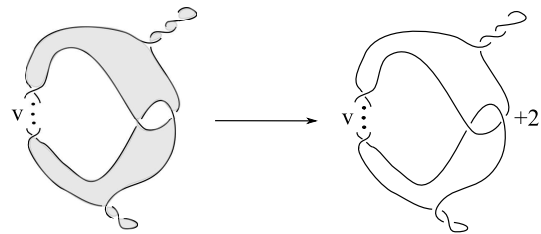}
\caption{\footnotesize{A surface $D$ with boundary $J$ is a \textit{spanning surface} for $J$ if $D$ is topologically equivalent to a doubly-punctured planar disc with twists when $J$ is a twist knot $C(2,v)$.}}
\label{spanningsurfacee} 
\end{center}
\end{figure}

\begin{figure}
\begin{center}
\includegraphics[width=10cm]{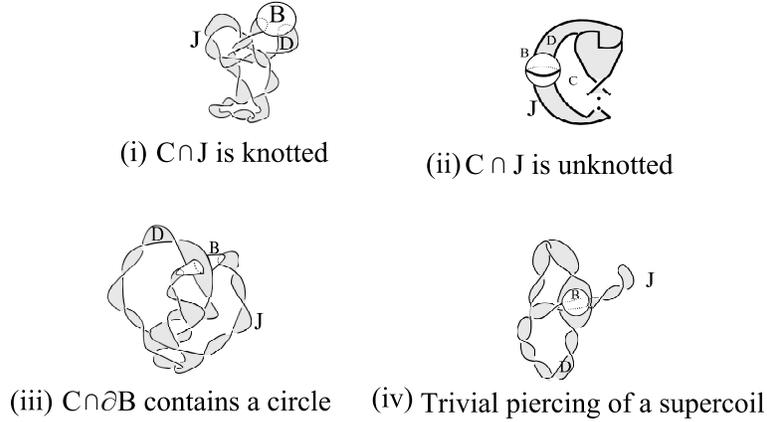}
\caption{\footnotesize{(i) A knot is trapped in the DNA branches outside of $B$. (ii) An unknotted substrate with the synaptic complex  formed. (iii) The recombinase complex pierces a supercoil in a non-trivial way, $D\cap \partial B$ contains at least one circle as well as two arcs. (iv) The productive synapse $B$ trivially pierces through a supercoil and the circle contained in $D\cap B$ can be removed via an isotopy of $C$. Scenarios (i) and (iv) are allowed by our assumptions, the other two scenarios are not.}}
\label{pierce} \label{knotted v unknoted} 
\end{center}
\end{figure}

\begin{figure}
\begin{center}
\includegraphics[width=15cm]{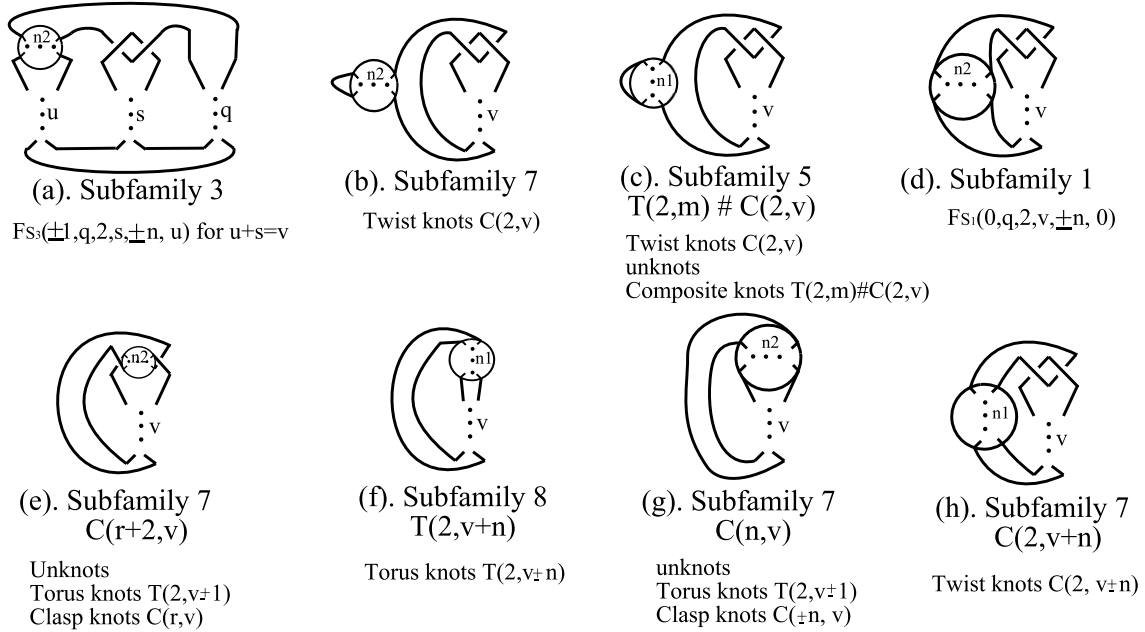}
\caption{\footnotesize{Products of recombination with serine recombinases, Theorem 2: All possible projections of the post-recombinant conformation of the recombinase-DNA complex $J\cup B$ and the productive synapse $B$ after $n$ rounds of processive recombination with a serine recombinase. The images inside the circles denote forms $n1$ and $n2$ of $B$ after $n$ rounds of processive recombination.}}
\label{serine products}
\end{center}
\end{figure}

\begin{figure}
\begin{center}
\includegraphics[width=15cm]{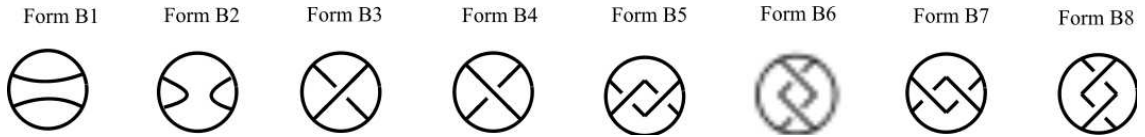} 
\caption{\footnotesize{Mathematical Assumption 3: Tyrosine recombinases. Projections of the possible post-recombinant conformations of the recombinase complex. Hooks are allowed because they have projections with at most one crossing between the sites.}}
\label{tyrosinesmaths}
\end{center}
\end{figure}

\begin{figure}
\begin{center}
\includegraphics[width=18cm]{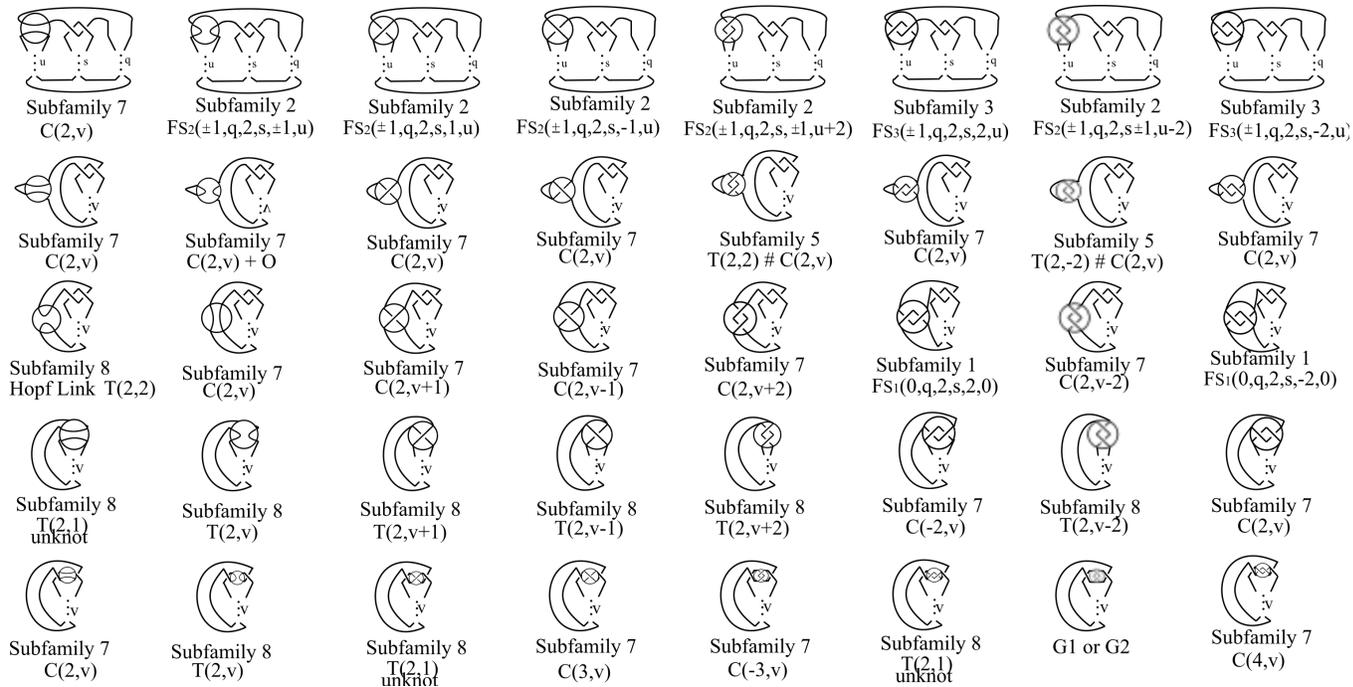}
\caption{\footnotesize{Products of recombination with tyrosine recombinases. Theorem 1: Projections of all possible conformations of the post-recombinant recombinase-DNA complex $J\cup B$ and the productive synapse $B$ after a reaction with a twist knot substrate $C(2,v), v\neq0$ mediated by a tyrosine recombinase.} }
\label{tyrosine products}
\end{center}
\end{figure}

\begin{figure}
\begin{center}
\includegraphics[width=18cm]{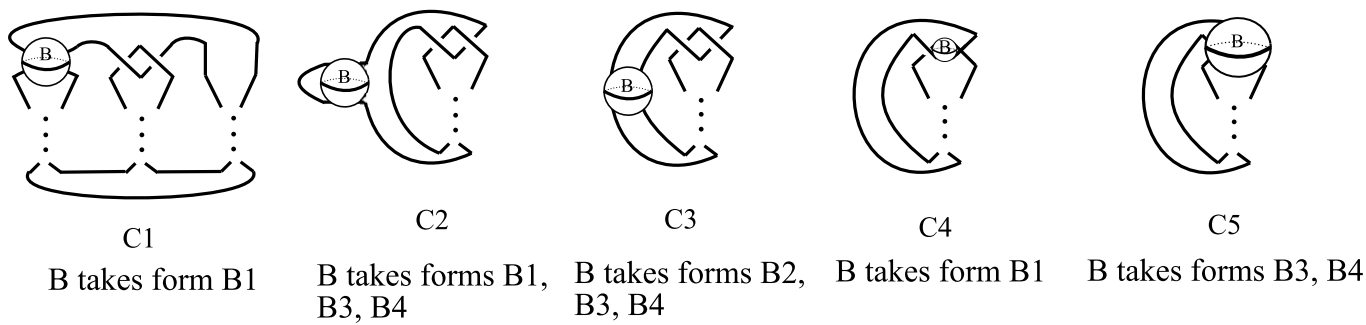} 
\caption{\footnotesize{Summary of Lemma 1. All the possible distinct forms that the substrate molecule $C\cap J$ can take, up to isotopy, along with the corresponding forms of $B$ for each form of $C\cap J$.}}
\label{all the possibilities for substrate}
\end{center}
\end{figure}

\begin{figure}
\begin{center}
\includegraphics[width=8cm]{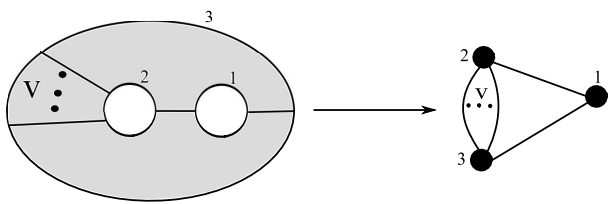} 
\caption{\footnotesize{A thrice-punctured $S^2$ in $S^3$, with arcs connecting the three punctures can be regarded as  a graph with three points and a collection of edges connecting them.}}
\label{S2}
\end{center}
\end{figure}

\begin{figure}
\begin{center}
\includegraphics[width=14cm]{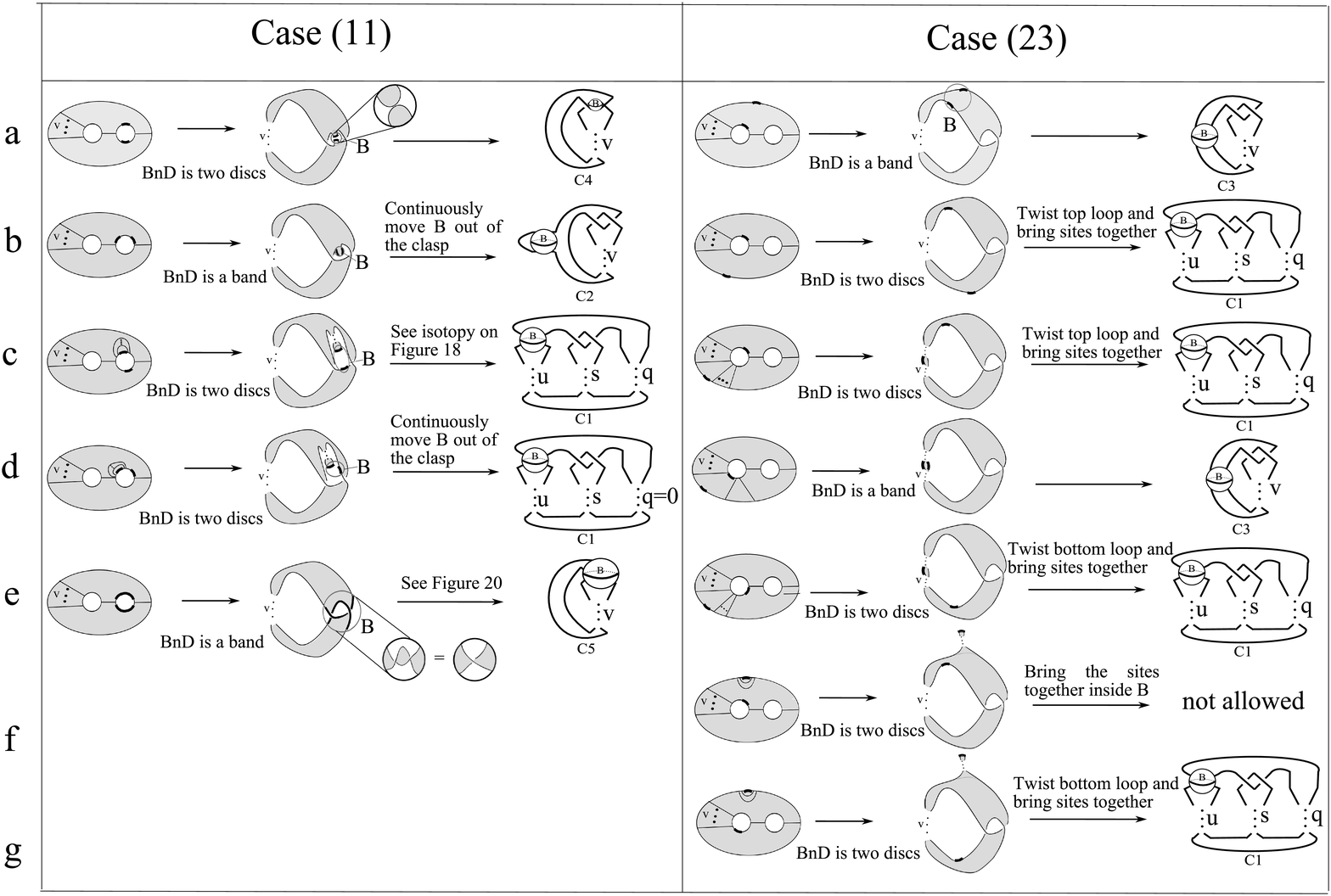} 
\caption{\footnotesize{Characterization of the recombinase-DNA complex. Specific sites situated either \textit{Case (11):} both on the boundary of puncture 1 of $S^2$, \textit{Case (23):} one on the boundary of puncture 2 of $S^2$ and the other on the boundary of puncture 3 of $S^2$. In cases where the right-most column says `not allowed', we mean that we can not allow such a conformation because when bringing the two specific sites together inside $B$, we get $C\cap D$ non-planar, which is not allowed by assumption 2.}}
\label{cases1}
\end{center}
\end{figure}

\begin{figure}
\begin{center}
\includegraphics[width=14cm]{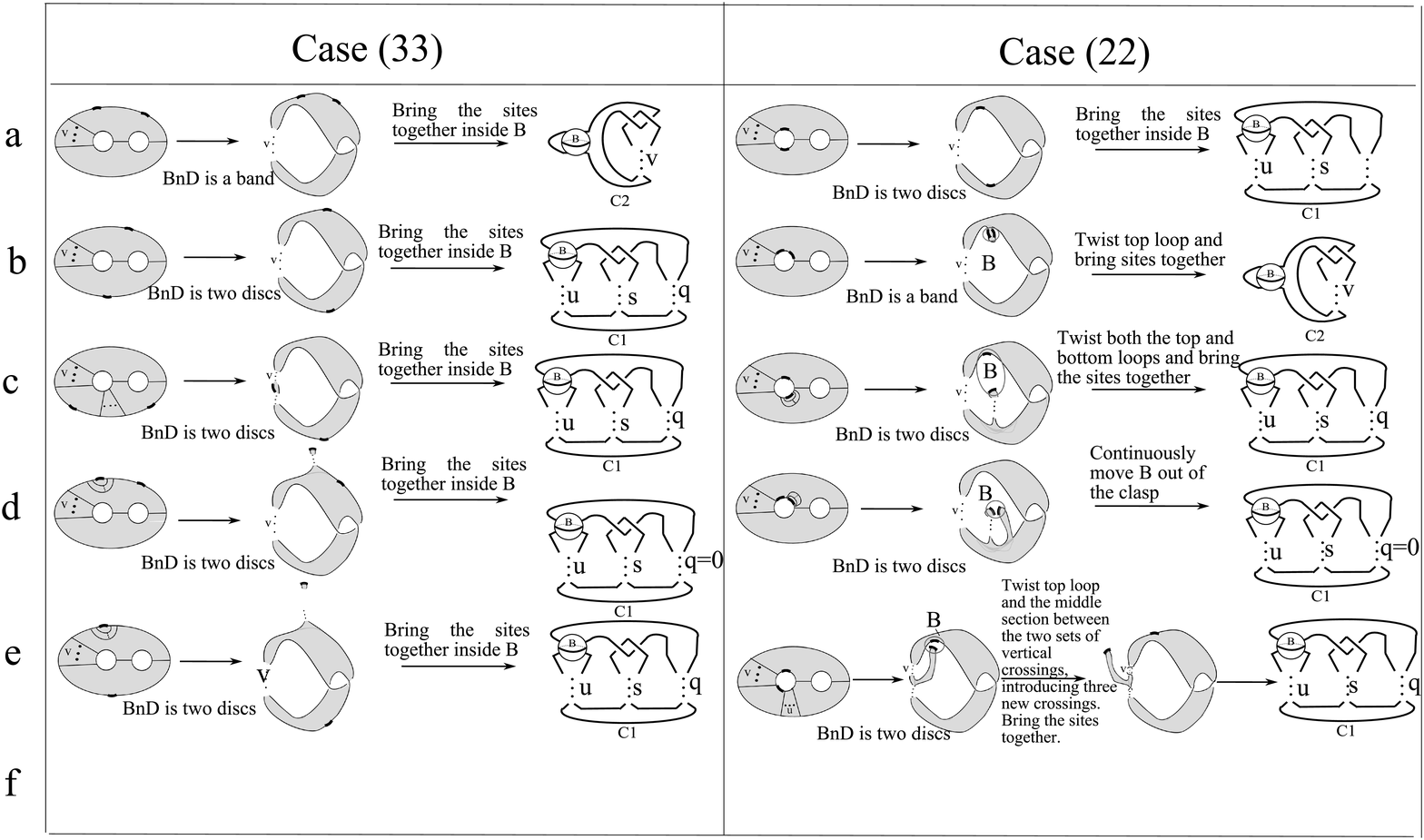} 
\caption{\footnotesize{Characterization of the recombinase-DNA complex. Specific sites situated either \textit{Case (22):} both on the boundary of puncture 2 of $S^2$,  \textit{Case (33):} both on the boundary of puncture 3 of $S^2$. Note that cases (22) and (33) are equivalent, but we consider both here because it may be more straightforward to visualise the isotopy of $C \cap J$ to one of the standard forms in one case or the other. }}
\label{cases3}
\end{center}
\end{figure}

\begin{figure}
\begin{center}
\includegraphics[width=16cm]{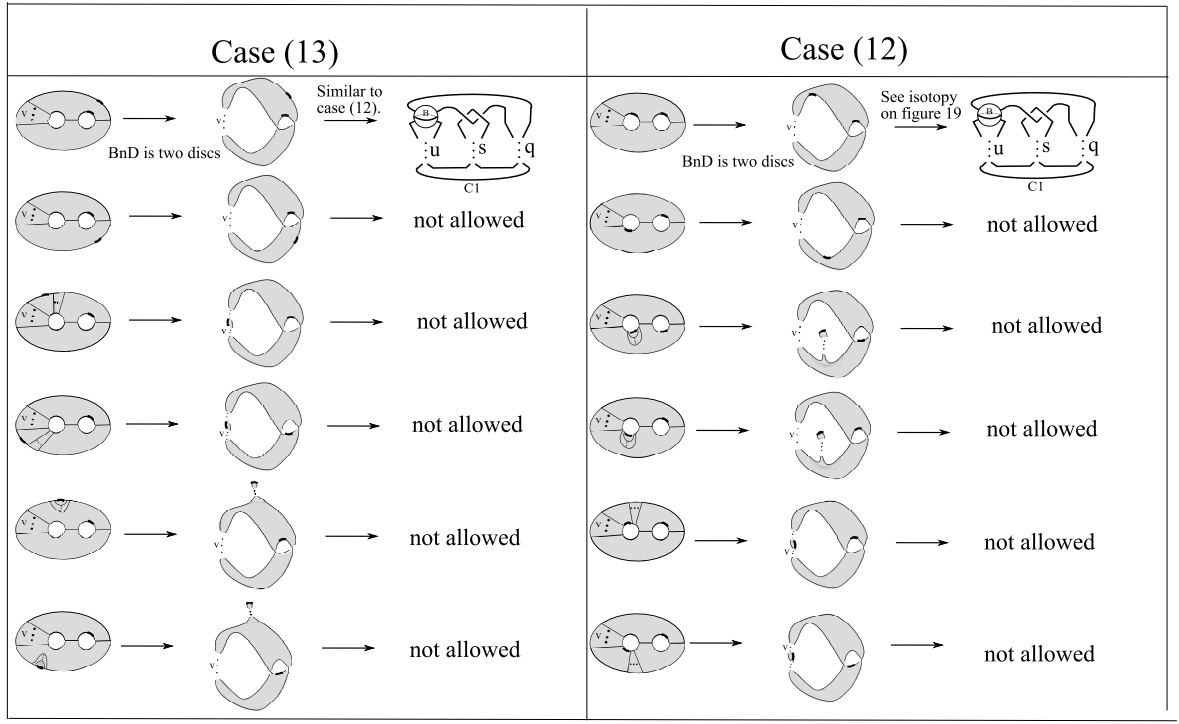} 
\caption{\footnotesize{Characterization of the recombinase-DNA complex. Specific sites situated either \textit{Case (12):} one on the boundary of puncture 1 of $S^2$ and the other on the boundary of puncture 2 of $S^2$, \textit{Case (13):} one on the boundary of puncture 1 of $S^2$ and the other on the boundary of puncture 3 of $S^2$. In cases where the right-most column says `not allowed', we mean that we can not allow such a conformation because when bringing the two specific sites together inside $B$, we get $C\cap D$ non-planar, which is not allowed by assumption 2. Note that cases (12) and (13) are equivalent, but we consider both here because it may be more straightforward to visualise the isotopy of $C \cap J$ to one of the standard forms in one case or the other. }}
\label{cases2}
\end{center}
\end{figure}

\begin{figure}
\begin{center}
\includegraphics[width=18cm]{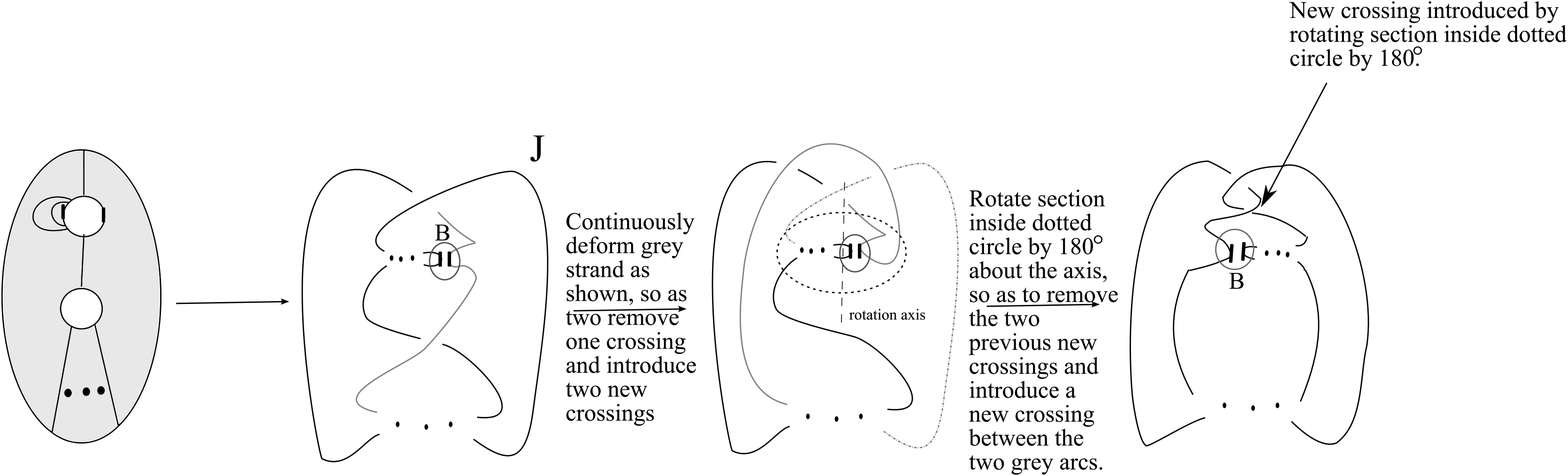} 
\caption{\footnotesize{Isotopy for \textit{Case (11c):} Assume both arcs of $D\cap\partial B$ lie on boundary 1 of the thrice-punctured $S^2\subset S^3$. The thrice-punctured $S^2$ generates the spanning surface $D$ that has boundary illustrated on the second image from the left. Next, a continuous deformation taking this conformation of $J$ to form C1.}}
\label{deformation11c}
\end{center}
\end{figure}

\begin{figure}
\begin{center}
\includegraphics[width=12cm]{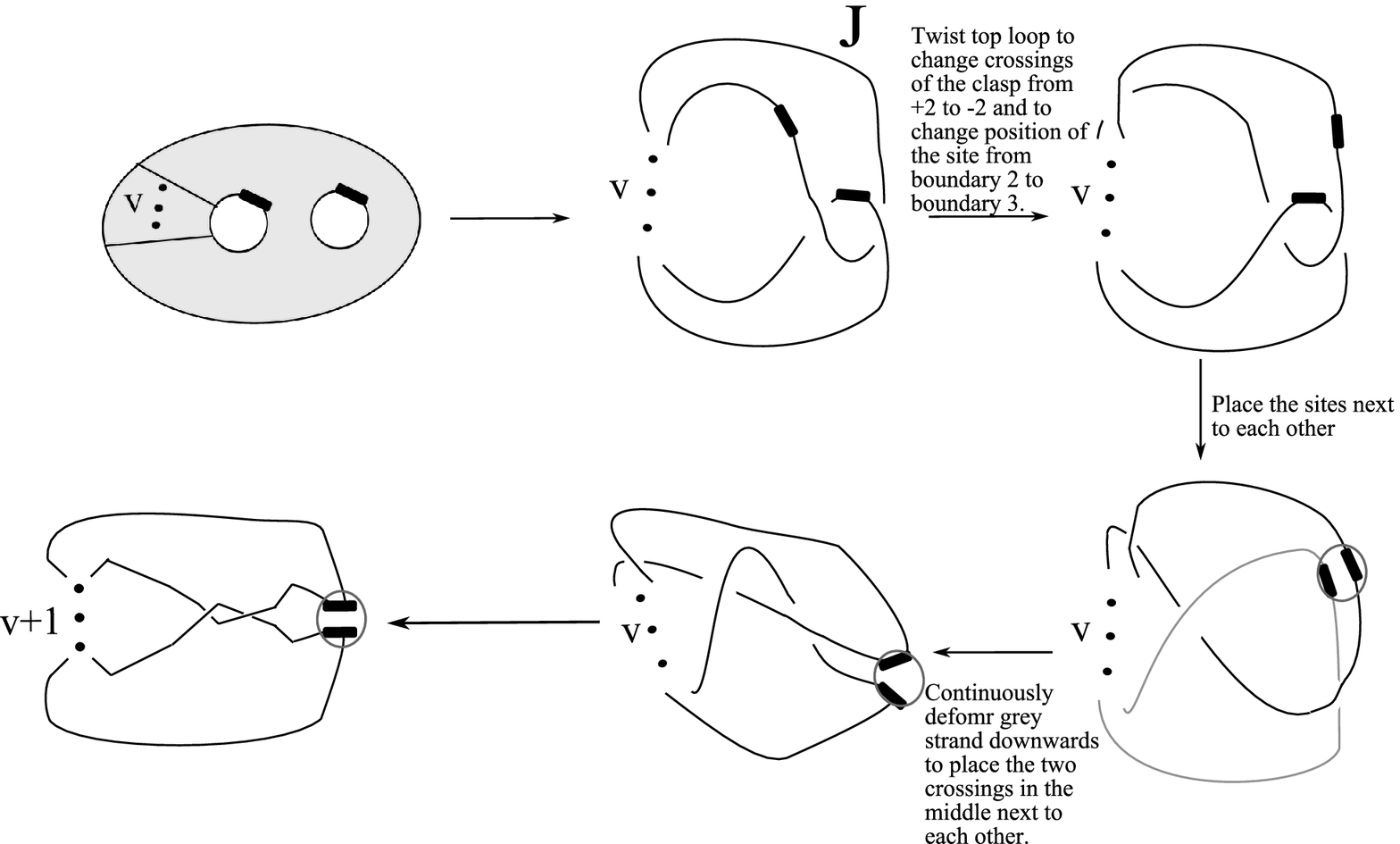} 
\caption{\footnotesize{Isotopy for \textit{Case (12a):} Assume both arcs of $D\cap\partial B$ lie on boundary 1 of the thrice-punctured $S^2\subset S^3$. The thrice-punctured $S^2$ generates the spanning surface $D$ that has boundary illustrated on the second image from the left. Next, a continuous deformation taking this conformation of $J$ to form C1.}}
\label{deformation12a}
\end{center}
\end{figure}

\begin{figure}
\begin{center}
\includegraphics[width=18.3cm]{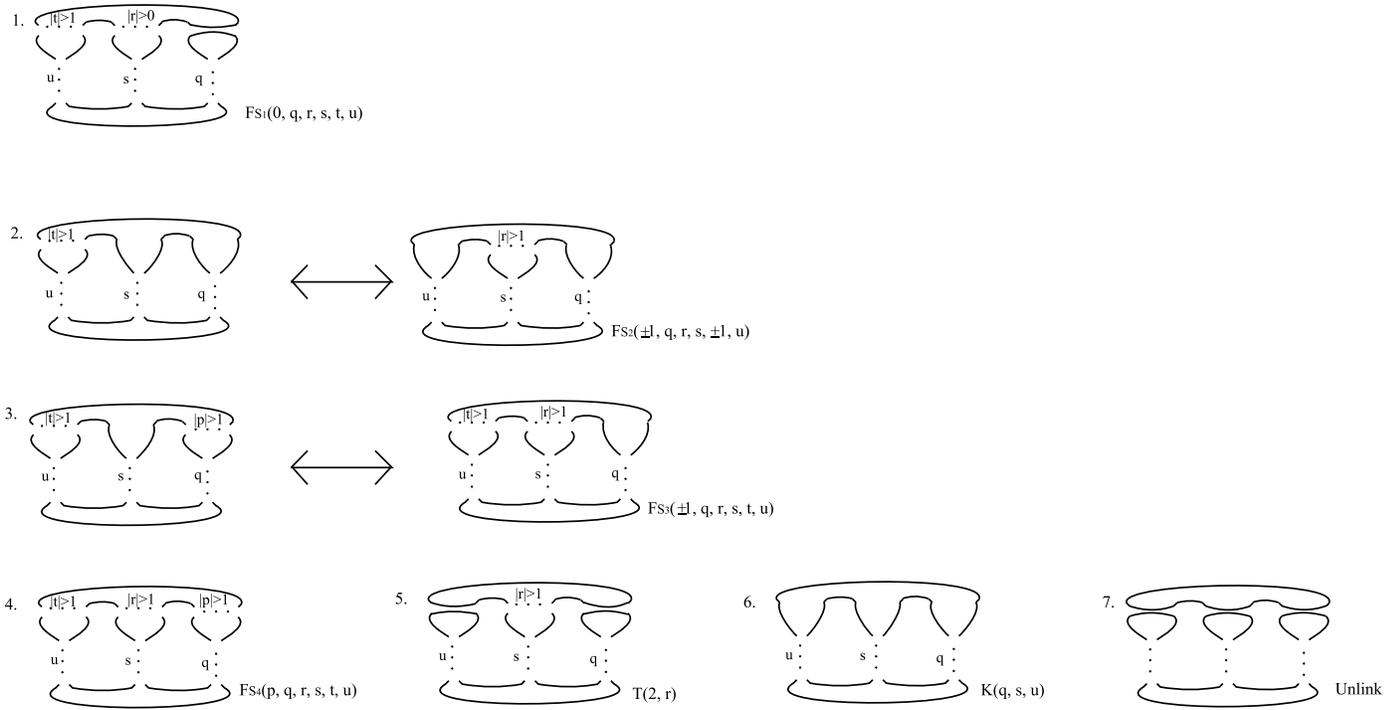} 
\caption{\footnotesize{Proof of Theorem 4: Our family in Figure \ref{family} was broken down into these subfamilies to be able to compute Tables 2,3,4 and 5.}}
\label{counting}
\end{center}
\end{figure}

\begin{figure}
\begin{center}
\includegraphics[width=15cm]{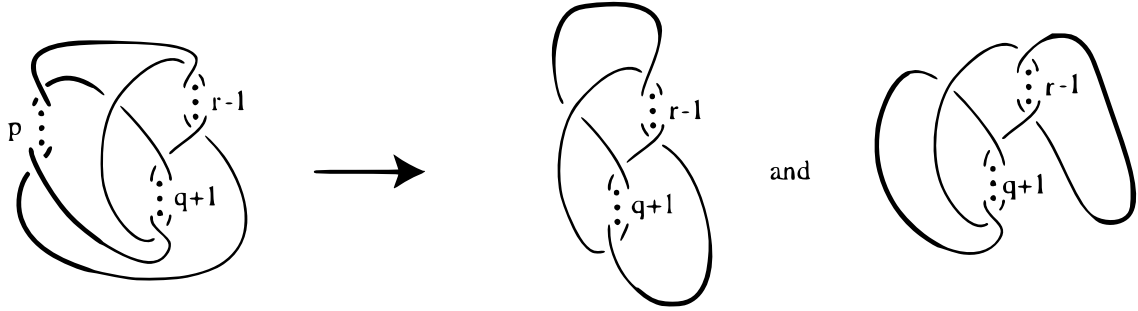} 
\caption{\footnotesize{Haya-Yamamoto: A projection of a knot or link is Hara-Yamamoto if when we cut off the row of $p$ crossings on the left and reseal the strands in the two natural ways then both resulting projections are reduced alternating.}}
\label{HY}
\end{center}
\end{figure}

\begin{figure}
\begin{center}
\includegraphics[width=15cm]{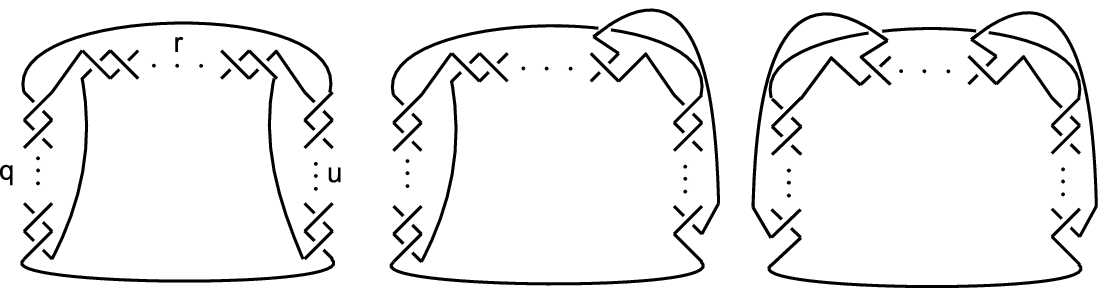} 
\caption{\footnotesize{Example of strand movement: By moving the two strands, we reduce from $|r|+|u|+|q|$ crossings originally, to $|r|+|u|+|q|-2$ crossings in the alternating diagram.}}
\label{strandmovement}
\end{center}
\end{figure}


\begin{thebibliography}{99}
\bibitem{MaxwellBates} A. Maxwell and A. D. Bates, DNA Topology, Oxford University Press (2005).
\bibitem{wang}  James C. Wang, Untangling the Double helix, Cold Spring Harbor Laboratory Press (2009).
\bibitem{mobileDNA} N. Craig, R. Craigie, M. Gellert and A. Lambowitz (ed)  Mobile DNA II (Washington, DC: ASM) (2002)
\bibitem{vinograd} J. Vinograd J., Lebowitz, R. Radloff, R. Watson, and P. Laipis, Twisted circular form of polyoma viral DNA. Proc. Natl. Acad. Sci. USA  (1965), 53, 1104-1111.
\bibitem{Wasserman-Cozzarelli1985} S. A. Wasserman and N. R. Cozzarelli, Determination of the stereostructure of the product of Tn3 resolvase by a general method., Proc Natl Acad Sci U S A. (1985), 82, 1079-1083.
\bibitem{Cozzarelli-et-al} N. R. Cozzarelli, M. A. Krasnow, S. P. Gerrard, J. H. White, A topological treatment of recombination and topoisomerases, Cold Spring Harb Symp Quant Biol. (1984) 49, 383-400.
\bibitem{4} S. A.Wasserman, J. M. Dungan and N. R. Cozzarelli, Discovery of a predicted DNA knot substantiates a model for site-specific recombination. Science (1985) 229, 171-174.
\bibitem{5} K. A. Heichman, I. P. Moskowitz and R. C. Johnson Configuration of DNA strands and mechanism of strand exchange in the Hin invertasome as revealed by analysis of recombinant knots. Genes Dev. (1991) 5, 1622-1634.
\bibitem{6} S. K. Merickel and R. C. Johnson, Topological analysis of Hin-catalysed DNA recombination in vivo and in vitro. Mol. Microbiol.  (2004) 51, 1143-1154.
\bibitem{7} R. Kanaar, A. Klippel, E. Shekhtman, J. M. Dungan, R. Kahmann and N. R. Cozzarelli,  Processive recombination by the phage Mu Gin system: implications for the mechanisms of DNA strand exchange, DNA site alignment, and enhancer action. Cell (1990) 62, 353-366.
\bibitem{8} N. J. Crisona, R. Kanaar, T. N. Gonzalez, E. L. Zechiedrich, A. Klippel and N. R. Cozzarelli,  Processive recombination by wild-type Gin and an enhancer independent mutant: Insight into the  mechanisms of recombination selectivity and strand exchange. J. Mol. Biol. (1994) 243, 437-457.
\bibitem{9} R. H. Hoess, A. Wierzbick, K. Abremski,  The role of the loxP spacer region in P1 site-specific recombination, Nucleis Acid Res. (1986) 14, 2287-2300.
\bibitem{11} Cox, DNA inversion in the 2$\mu$m plasmid of \textit{Saccaromyces cervisiae}, In Mobile DNA pp. 661-670, American Society for microbiology, Washington, DC.  (1989)
\bibitem{12} S. J. Spengler, A. Stasiak, N. R. Cozzarelli, The stereostructure of knots and catenanes produced by phage $\lambda$ integrative recombination: Implications for mechanism and DNA  structure, Cell  (1985) 42, 325-334.
\bibitem{13} S. D. Colloms, J. Bath, D. J. Sherratt, Topological selectivity in Xer site-specific recombination, Cell (1997) 88, 855-864.
\bibitem{10} N. J. Crisona,  R. L. Weinberg, B. J. Peter, D. W. Sumners, and N. R. Cozzarelli, The topological mechanism of phage lambda integrase. J. Mol. Biol. (1999) 289, 747-775.
\bibitem{Sumners-et-al} D. W. Sumners, C. Ernst, S. J. Spengler, N. R. Cozzarelli, Analysis of the mechanism of DNA recombination using tangles. Quart. Review Bioph. (1995), 28, 253-313.
\bibitem{ES} C. Ernst and D. W. Sumners, A calculus for rational tangles: applications to DNA recombination Math. Proc. Camb. Phil. Soc., (1990), 108, 489-515
\bibitem{darcy} I. K. Darcy, Biological distances on DNA knots and links: applications to Xer recombination Knots in Hellas '98, J. Knot Theory Ramifications, (2001), 10, 269-94
\bibitem{vazquez-gin} M. Vazquez, D. W. Sumners, Tangle analysis of Gin site-specific recombination, Math. Proc. Camb. Phil. Soc. (2004) 136,  565-582
\bibitem{vazquez-et-al} M. Vazquez, S. D. Colloms and D. W. Sumners, 3-Tangle Analysis of Xer Recombination Reveals only Three Solutions, all Consistent with a Single Three-dimensional Topological Pathway, J. Mol. Biol. (2005) 346, 493-504
\bibitem{d3} I. K. Darcy, Modeling protein-DNA complexes with tangles, Comput. Math. Appl. (2008), 55 , 924--937
\bibitem{d4} A. A. Vetcher, A. Y. Lushnikov, J. Navarra-Madsen, R. G. Scharein, Y. L. Lyubchenko, I. K. Darcy, S. D. Levene, DNA Topology and Geometry in Flp and Cre Recombination, Journal of Molecular Biology, (2006), 4, 1089-104.
\bibitem{v1} Y. Saka, M. Vazquez, TangleSolve: Topological analysis of site-specific recombination, Bioinformatics, (2002), 18, 1011-1012
\bibitem{k1} I. K. Darcy, K. Ishihara, R. Medikonduri and K. Shimokawa, Rational tangle surgery and Xer recombination on catenanes, preprint (2009)
\bibitem{k2} K. Shimokawa, K. Ishihara and M. Vazquez, Tangle analysis of DNA unlinking by the Xer/FtsK system, to appear in Bussei Kenkyu (2009), Proceedings of the International Conference "Knots and soft-matter physics: topology of polymers and related topics in physics, mathematics and biology".
\bibitem{d5} I. K. Darcy, D. W. Sumners, Rational Tangle Distances on Knots and Links, Mathematical Proceedings of the Cambridge Philosophical Society, 128 (2000), no. 3, 497--510.
\bibitem{15} S. Trigueros, J. Arsuaga, M. E. Vazquez, D. W. Sumners and J. Roca, Novel display of knotted DNA molecules by two-dimensional gel electrophoresis. Nucleic Acids Res. (2001), 29, E67.
\bibitem{14} M. A. Krasnow, A. Stasiak, S. J. Spengler, F. Dean, T. Koller and N. R. Cozzarelli, Determination of the absolute handedness of knots and catenanes of DNA. Nature,  (1983), 304, 559-560.
\bibitem{16} L. E. Zechiedrich and N. J. Crisona, Coating DNA with RecA protein to distinguish DNA path by electron microscopy. In Methods in Molecular Biology: DNA Topoisomerase Protocols (Bjornsti, M. and Osheroff, N., eds),  (1989), 1, 98-108, Humana Press, Totowa, NJ.
\bibitem{jonespoly} V.F.R. Jones, A polynomial invariant for knots via von Neumann algebra. Bull. Amer. Math. Soc.(N.S.) (1985) 12, 103-111
\bibitem{numberofknots} J. Hoste, M. Thistlethwaite and J. Weeks, The first 1, 701, 936 knots Math. Intelligencer, (1998), 20, 33-48
\bibitem{bio} N. D.F. Grindley,  K. L. Whiteson,  and P. A. Rice, Mechanisms of Site-Specific Recombination, Annu. Rev. Biochem. (2006), 75, 567-605
\bibitem{mouse} R. Feil, Conditional Somatic Mutagenesis in the Mouse Using Site-Specific Recombinases, Springer-Verlag Berlin Heidelberg (2007) HEP 178, 3-28.
\bibitem{tools} N. J. Kilby, M. R. Snaith and James A. H. Murray, Site-specific recombinases: tools for genome engineering, Trends in Genetics (1993), 9, Issue 120, 413-421
\bibitem{SECS}
D.W. Sumners, C. Ernst, S. J. Spengler and N. R. Cozzarelli,
\textit{Analysis of the mechanism of {D}{N}{A} recombination using
tangles}, Quarterly Review of Biophysics \textbf{28} no. 3 (1995)
253--313.
\bibitem{BFmaths} D. Buck and E. Flapan, A topological characterization of knots and links arising from site-specific recombination, J. Phys. A: Math. Theor. 40 (2007) 12377-12395.
\bibitem{BFbio} D. Buck and E. Flapan, Predicting Knot or Catenane Type of Site-Specific Recombination Products, J. Mol. Biol. (2007) 374, 1186-1199.
\bibitem{KDbio} K. Valencia and D. Buck, Predicting knot and catenanes type of site-specific recimbination products of twist knot substrates, arXiv:1007.2513v1 q-bio.QM.
\bibitem{Rolfsen} D. Rolfsen, Knots and Links, AMS (2003).
\bibitem{BZ} G. Burde, H. Zieschang, Knots, de Gruyter Studies in Mathematics (2003).
\bibitem{Cromwell} P. Cromwell, Knots and Links, Cambridge University Press (2005).
\bibitem{Kawauchi} A. Kawauchi, A survey of Knot Theory, Birkhauser (1996).
\bibitem{Murasugi} K. Murasugi, Knot theory and its applications, Modern Birkhauser Classics (2008).
\bibitem{Conwaytangle} J. H. Conway. An enumeration of knots and links, Computational problems in abstract algebra  (ed. J. Leech), Pergamon Press (1969), 329-358.
\bibitem{CS} C. Ernst and D. W. Sumners,  The growth in the number of prime knots Math. Proc. Camb. Phil. Soc., (1987), 102, 303-15
\bibitem{Murasugi1} K. Murasugi, Jones polynomials and classical conjectures in knot theory Topology, (1987), 26, 187-94
\bibitem{Thist} W. B. R. Thistlethwaite,  A spanning tree expansion of the Jones polynomial Topology, (1987), 26, 297-309
\bibitem{Hara-Yamamoto} M. Hara and M. Yamamoto, Some links with nonadequate minimal-crossing number Math. Proc. Camb. Philos. Soc., (1992), 111, 283-9
\bibitem{LickThist}W. B. R. Lickorish and M. B. Thistlethwaite, Some links with nontrivial polynomials and their crossing number Comment. Math. Helvetici, (1988), 63, 527-39






\end{thebibliography}
\end{document}